\documentclass[11pt,a4paper]{article}
\usepackage{t1enc}
\usepackage[utf8]{inputenc}
\usepackage{amsfonts}
\usepackage{amsthm}
\usepackage{amsmath}
\usepackage{amscd}
\usepackage[mathscr]{eucal}
\usepackage{indentfirst}
\usepackage{graphicx}
\usepackage{graphics}
\usepackage{hyperref}
\usepackage{pict2e}
\usepackage{epic}
\usepackage{pgfplots}
\usepackage{comment}
\pgfplotsset{compat=1.15}
\usepackage{mathrsfs}
\usepackage{todonotes}
\usepackage[margin=2.5cm]{geometry}
\usepackage{epstopdf} 
\usepackage{todonotes}

\theoremstyle{plain}
\newtheorem{theorem}{Theorem}[section]
\newtheorem{lemma}[theorem]{Lemma}

\newtheorem{proposition}[theorem]{Proposition}

\theoremstyle{definition}
\newtheorem{definition}[theorem]{Definition}
\newtheorem{conjecture}[theorem]{Conjecture}

\newtheorem{remark}[theorem]{Remark}
\newtheorem{?}[theorem]{Problem}

\newcommand{\eps}{\varepsilon}
\newcommand{\zz}{\mathbb{Z}}
\newcommand{\ff}{\mathbb{F}}

\begin{document}
	\renewcommand{\refname}{References}
	\pagenumbering{arabic}

	\newpage
	\setcounter{page}{1}
	
	\title{Finding a perfect matching of $\ff_2^{\lowercase{n}}$ with prescribed differences}
	\author{Benedek Kovács\thanks{ELTE Linear Hypergraphs Research Group, Eötvös Loránd University, 1-3 Egyetem tér, 1053, Budapest, Hungary. Email address: \url{benoke98@student.elte.hu}}}
        \date{September 3, 2024}
	
	\maketitle
	
	\begin{abstract}
		We consider the following question by Balister, Győri and Schelp: given $2^{n-1}$ nonzero vectors in $\ff_2^n$ with zero sum, is it always possible to partition the elements of $\ff_2^n$ into pairs such that the difference between the two elements of the $i$-th pair is equal to the $i$-th given vector for every $i$? An analogous question in $\ff_p$, which is a case of the so-called "seating couples" problem, has been resolved by Preissmann and Mischler in 2009. In this paper, we prove the conjecture in $\ff_2^n$ in the case when the number of distinct values among the given difference vectors is at most $n-2\log n-1$, and also in the case when at least a fraction $\frac12+\eps$ of the given vectors are equal (for all $\eps>0$ and $n$ sufficiently large based on $\eps$). 
	\end{abstract}
	
	\section{Introduction}	
	We consider the following conjecture of Balister, Győri and Schelp \cite{Balister_Gyori_Schelp} from 2008:
	
	\begin{conjecture}[\textit{Main conjecture}]\label{fosejtes} Let $n\ge 2$ be an integer and $m=2^{n-1}$. If the nonzero difference vectors $\boldsymbol{d}_1$, $\boldsymbol{d}_2$, \dots, $\boldsymbol{d}_m$ are given in $\ff_2^n$ such that $\sum\limits_{i=1}^m \boldsymbol{d}_i=\boldsymbol{0}$ (and the $\boldsymbol{d}_i$'s are not necessarily distinct), then $\ff_2^n$ can be subdivided into disjoint pairs $\{\boldsymbol{a}_i, \boldsymbol{b}_i\}$ ($1\le i\le m$) such that $\boldsymbol{a}_i-\boldsymbol{b}_i=\boldsymbol{d}_i$ holds for every $i$.
	\end{conjecture}
	
	\textit{Here and throughout the article, $\ff_q$ denotes the $q$-element field where $q$ is a prime power, and $\ff_q^n$ denotes the vector space of dimension $n$ over $\ff_q$.}
	
	In the same year, Bacher has independently posed another analogous version of this question, where instead of $\ff_2^n$, we partition the elements of $\ff_p\setminus\{0\}$ into pairs, where $p$ is an odd prime, and there is no restriction on the sum of the given (nonzero) differences. For this case, Preissmann and Mischler gave a positive answer \cite{Preissmann_Mischler}; their method relies on summing the values of an appropriate multivariate polynomial over $\ff_p$.
	
	\begin{theorem}[Preissmann, Mischler]\label{PM_tetel}
		Let $p$ be an odd prime and $M=\frac{p-1}{2}$. If in $\ff_p$, the nonzero differences $d_1, d_2, \dots, d_M$ are given, then $\ff_p\setminus \{0\}$ can be partitioned into disjoint pairs $\{a_i, b_i\}$ ($1\le i\le M$) such that $a_i-b_i=d_i$ for each $i$.
	\end{theorem}
	
	Later, Kohen and Sadofschi \cite{Kohen_Sadofschi} gave a new proof of this claim using the Combinatorial Nullstellensatz.
	
	The statement can also be investigated for other cyclic groups as well. In the ring $\zz_n$ of mod $n$ residue classes, let  $\zz_n^{*}$ denote the set of units, i.e. the elements coprime to $n$. The following conjecture of Adamaszek pertaining to cyclic groups of even order has been proven by Kohen and Sadofschi \cite{Kohen_Sadofschi_DM}:
	
	\begin{theorem}[Kohen, Sadofschi]
		Let $n=2M$ be a positive even number. If the elements $d_1, d_2, \dots, d_M\in \zz_n^{*}$ are arbitrarily given, then $\zz_n$ can be partitioned into disjoint pairs $\{a_i, b_i\}$ such that $a_i-b_i=d_i$ for each $i$.
	\end{theorem}
	
	For newer results on this case, see \cite{Meszka_Pasotti_Pellegrini}. In cyclic groups of odd order, Pasotti and Pellegrini \cite{Pasotti_Pellegrini} made the following conjecture, which would generalize Theorem \ref{PM_tetel}:
	
	\begin{conjecture}[Pasotti, Pellegrini]
		Let $n=2M+1$ be a positive odd number. If the elements $d_1, d_2, \dots, d_M\in \zz_n\setminus \{0\}$ are arbitrarily given, then $\zz_n\setminus \{0\}$ can be partitioned into disjoint pairs $\{a_i, b_i\}$ such that $a_i-b_i=d_i$ for each $i$ if and only if the following condition holds: 
		
		\begin{center}
			\begin{minipage}{0.6\textwidth}
				\centering
			for any divisor $k$ of $n$, the number of multiples of $k$ among the differences $d_i$ does not exceed $\frac{n-k}{2}$.
			\end{minipage}
		\end{center}
	\end{conjecture}

	We note that they formulated the conjecture using the terminology of matchings, which can be generally used to describe problems of this type: assign a colour to each nonzero element of an abelian group $G$, with $g,g'\in G$ getting the same colour if and only if $g'=\pm g$. In the complete graph with vertex set $G$, each edge $ab$ is assigned the colour of $a-b$. Then we would like to find a perfect matching (if $|G|$ is even) or a matching covering all but one vertices (if $|G|$ is odd) such that the colours of the matching's edges form a prescribed multiset.
	
	Another way to generalize Theorem \ref{PM_tetel} is if we consider the problem for $\ff_p^n\setminus \{\mathbf{0}\}$ instead of $\ff_p\setminus \{0\}$. Karasev and Petrov \cite{Karasev_Petrov} showed that in this case, the same statement does not hold: for example if every difference $\mathbf{d}_i$ is equal to the same nonzero vector $\mathbf{d}$, then in each vector pair, the two elements have to fall into the same $\langle \mathbf{d}\rangle$-coset, however the nonzero cosets have an odd number of elements in $\ff_p^n\setminus \{\mathbf{0}\}$, hence they cannot be partitioned into such pairs. However they have shown the following result \cite[Theorem 3]{Karasev_Petrov}:
	
	\begin{theorem}[Karasev, Petrov]
		Let $p$ be an odd prime and $M=\frac{p^n-1}{2}$. If the sets $\{\boldsymbol{d}_{1,1}, \dots, \boldsymbol{d}_{1,n}\}$,  $\{\boldsymbol{d}_{2,1}, \dots, \boldsymbol{d}_{2,n}\}$, \dots, $\{\boldsymbol{d}_{M,1}, \dots, \boldsymbol{d}_{M,n}\}$ are given in $\ff_p^n$ such that each set is a linear basis of $\ff_p^n$, then there exists a function $g: \{1, ..., M\}\to \{1, ..., n\}$ such that $\ff_p^n\setminus \{\boldsymbol{0}\}$ can be subdivided into disjoint pairs $\{\boldsymbol{a}_i, \boldsymbol{b}_i\}$ ($1\le i\le M$) such that $\boldsymbol{a}_i-\boldsymbol{b}_i=\boldsymbol{d}_{i,g(i)}$ for each $i$.
	\end{theorem}
	
	If we investigate the statement in $\ff_2^n$ instead of $\ff_p^n$, and would like a perfect matching, we also need to include the zero vector in the set of elements to be matched. Even in this case, the claim does not hold for arbitrary nonzero difference vectors, as the sum of the differences has to be equal to the sum of all elements of the vector space, which is zero (see Lemma \ref{vektorterben_osszeg0}). By the main conjecture, this would be a sufficient condition in order for an adequate perfect matching to exist.
	
	The authors of \cite{Balister_Gyori_Schelp} have also verified the conjecture for the case $n\le 5$, and they have proved the main conjecture in the special case where among the vectors $\mathbf{d}_i$, there is a value appearing at least $\frac12m$ times, given also that all values appear an even number of times \cite[Theorem 4]{Balister_Gyori_Schelp}:
	
	\begin{theorem}[Balister, Győri, Schelp] \label{feleazonos_tobbiparokban}
		The main conjecture is true in the case when the vectors $\boldsymbol{d}_1, \boldsymbol{d}_2, \dots, \boldsymbol{d}_{\frac{m}{2}}$ are all equal, and for every integer $\frac{m}{4}< i\le \frac{m}{2}$ we have $\boldsymbol{d}_{2i-1}=\boldsymbol{d}_{2i}$.
	\end{theorem}
	
	The main idea of their proof is to first greedily select the values $\mathbf{a}_i$ and $\mathbf{b}_i$ belonging to the differences whose value is not equal to $\mathbf{d}_1$, such that if in one step we selected the pair $(\mathbf{a}_i,\mathbf{b}_i)$, then in the next step we will select the pair $(\mathbf{a}_i+\mathbf{d}_1, \mathbf{b}_i+\mathbf{d}_1)$ which has the same difference. Then in the end, the remaining elements will form pairs of difference $\mathbf{d}_1$.
	\medskip
	
	In coding theoretical terms, Conjecture \ref{fosejtes} is referred to as the \textit{Hadamard solution} version of the \textit{functional batch code conjecture}. Under this framework, Hollmann, Khathuria, Riet and Skachek \cite{Hollmann_Khathuria_Riet_Skachek}, using a theorem of Hall \cite{Hall}, proved that Conjecture \ref{fosejtes} is true whenever all differences $\mathbf{d}_i$ lie outside a given hyperplane $H$ of $\ff_2^n$. For further related work, see \cite{Yohananov_Yaakobi__Hadamard} and \cite{Yohananov_Yaakobi__extendedSimplex} by Yohananov and Yaakobi.
	\medskip
	
	In 2023, Correia, Pokrovskiy and Sudakov \cite{Correia_Pokrovskiy_Sudakov} published the following result:
	
	\begin{theorem}[Correia, Pokrovskiy, Sudakov]\label{Correia_thm}
		Let $G$ be a multigraph whose edges are $t$-coloured, so that each colour class is a matching of size at least $t+20t^{15/16}$. Then there exists a rainbow matching of size $t$ (that is, a matching with $t$ edges of all distinct colours).
	\end{theorem}
	
	Applying this result to the graph on the vertex set $\ff_2^n$ with colour class $i$ consisting of the edges between pairs of difference $\mathbf{d}_i$, we get that for any $M\le 2^{n-1}-O(2^{15n/16})$ nonzero differences $\mathbf{d}_i$, we can find disjoint pairs $\{\mathbf{a}_i, \mathbf{b}_i\}$ such that $\mathbf{a}_i-\mathbf{b}_i=\mathbf{d}_i$. However this method does not result in perfect matchings.
	
	Gao, Ramadurai, Wanless and Wormald \cite{Gao_Ramadurai} conjectured that Theorem \ref{Correia_thm} holds for $t+2$ in place of $t+20t^{15/16}$, which would resolve this problem for any $M\le 2^{n-1}-2$ nonzero difference vectors.
	\medskip
	
	The main results of this paper are the following:
	
	\begin{theorem}\label{thm-n-2logn-1-osztaly}
		The main conjecture is true in the case when the number of distinct values among the difference vectors is at most $n-2\log n-1$.
	\end{theorem}
	
	\begin{theorem}\label{mainthm_halfpluseps}
		For every $\eps>0$, there exists a value $n_0$ such that for all $n\ge n_0$, the main conjecture is true in the case when at least a fraction $\frac12+\eps$ of the differences are all equal.
	\end{theorem}
	
	In the context of both main results, the value that appears the most commonly among the prescribed differences will be denoted by $\mathbf{u}$. (For Theorem \ref{thm-n-2logn-1-osztaly}, if multiple values tie for the highest number of occurrences, then pick any one of them to be $\mathbf{u}$.)
	
	In Section 2 we prove Theorem \ref{thm-n-2logn-1-osztaly} by looking for so-called \textit{circuits} among the differences, meaning inclusionwise minimal, modulo $\mathbf{u}$ linearly dependent sets. From these circuits we create multisets which we call $\mathbf{u}$-good. Using these we can build partial matchings so that among the remaining differences, each value will occur an even number of times.  Then the problem can be reduced to a much simpler task where there are only two distinct difference values.
	
	In Section 3 we prove Theorem \ref{mainthm_halfpluseps}. In order to do this, we use a result of Bonin and Qin \cite{Bonin_Qin} to partition almost all of the given differences not equal to $\mathbf{u}$ into affine subspaces, which serve as another important building block in our construction of the perfect matching. Then we partition the rest into $\mathbf{u}$-good multisets. We construct our matching by first using the $\mathbf{u}$-good multisets greedily, then using the affine subspaces which are more versatile, then completing the matching using any pairs of identical differences, and finally with the copies of $\mathbf{u}$.
	
	\subsection{Notation and preliminaries}
	
	Throughout the paper, $[n]$ denotes the set $\{1, 2, \dots, n\}$. We will also use the notation $A\triangle B$ for the symmetric difference of two sets $A$ and $B$. In an abelian group $G$, if $g\in G$ and $X\subseteq G$ then $g+X=\{g+x: x\in X\}$, and such a set is referred to as a \textit{translate} of $X$.
	
	When working with sequences of elements in an abelian group $G$, the concatenation of sequences $A$ and $B$ will be denoted by $A.B$, so if $A=(a_1,a_2,...,a_k)$ and $B=(b_1,b_2,...,b_{\ell})$ then $A.B=(a_1,...,a_k,b_1,...,b_{\ell})$.
	
	We will work with multisets as well, and this concatenation operator will also be defined for multisets (i.e., for every $c$, if element $c$ appears $m_1$ times in the multiset $M_1$ and $m_2$ times in the multiset $M_2$, then $c$ appears $m_1+m_2$ times in the multiset $M_1.M_2$). The \textit{size} of a multiset $M$ refers to the total number of elements in $M$, counted with multiplicities.
	
	The notation $A\sim B$ will mean that the sequences $A$ and $B$ can be obtained from each other by permuting their elements. If $A=(a_1, a_2, ..., a_k)$ and $S=\{i_1, i_2, ..., i_m\}\subseteq [k]$ with $i_1<i_2<...<i_m$, then $A_S=(a_{i_1}, a_{i_2}, ..., a_{i_m})$. If $A$ is a sequence or a multiset, we let $\sum A$ denote the sum of the elements of $A$, and say that $A$ is \textit{zero-sum} if $\sum A=0$.
	
	When a nonzero element $\mathbf{u}\in \ff_2^n$ is fixed, we will often work in the quotient space $\ff_2^n/\langle \mathbf{u}\rangle$, which is a vector space over $\ff_2$ of dimension $n-1$, and its elements are the $\langle \mathbf{u}\rangle$-cosets of $\ff_2^n$. (Each $\langle \mathbf{u}\rangle$-coset is an unordered pair of elements having a difference of $\mathbf{u}$, and $\ff_2^n$ can be partitioned into $2^{n-1}$ such pairs.) Two elements of $\ff_2^n$ are said to be \textit{congruent mod $\boldsymbol{u}$} if they lie in the same $\langle \mathbf{u}\rangle$-coset, i.e. they are equal or have a difference of $\mathbf{u}$.
	
	We will often refer to the following basic fact about $\ff_2^n$:
	
	\begin{lemma} \label{vektorterben_osszeg0}
		The sum of the elements of the vector space $\ff_2^n$ is $\boldsymbol{0}$ if $n\ge 2$.
	\end{lemma}
	
	\begin{proof}
		For a given $1\le i\le n$, there are $2^{n-1}$ vectors in $\ff_2^n$ whose $i$-th coordinate is 0 and there are $2^{n-1}$ whose $i$-th coordinate is 1. So modulo 2, in the sum of all vectors the $i$-th coordinate is $2^{n-1}=0$ since $n\ge 2$.
	\end{proof}
	
	\section{Perfect matching in the case of few difference classes}
	
	In this section, we will resolve the main conjecture in a special case. Let the nonzero differences $\mathbf{d}_1$, $\mathbf{d}_2$, \dots, $\mathbf{d}_m$ be given in $\ff_2^n$ such that $\sum_{i=1}^m \mathbf{d}_i=\mathbf{0}$, where $m=2^{n-1}$. The collections containing all differences equal to a fixed vector $\mathbf{d}$ will be called \textit{difference classes}. For a given configuration $(\mathbf{d}_1, \mathbf{d}_2, \dots, \mathbf{d}_m)$ of differences, let $t$ denote the number of nonempty difference classes. For a given value of $n$, we would like to determine the largest integer $T(n)$ for which we can guarantee the existence of a suitable perfect matching of $\ff_2^n$ when $t\le T(n)$.
	
	In the case $t=1$ the task is trivial, since in this case we have to divide our vector space into $2^{n-1}$ pairs of difference $\mathbf{u}$ for some $\mathbf{u}\ne \mathbf{0}$, so we can just take the $\langle \mathbf{u}\rangle$-cosets of $\ff_2^n$.
	
	In the case $t=2$, the task can be solved using Theorem \ref{feleazonos_tobbiparokban}, as $\sum \mathbf{d_i}=\mathbf{0}$ means that both difference classes have even size, and at least half of the differences have to be in the same class. So we have the structure that half of the differences are the same and the rest of the differences can be partitioned into equal-valued pairs.
	
	Our first main result, Theorem \ref{thm-n-2logn-1-osztaly}, states that a suitable perfect matching exists for $t\le n-2\log n-1$. In order to show this result, we will first need some definitions and a technical statement about circuits and simple orderings of sequences in abelian groups.
	
	\subsection{Circuits and simple orders}
	
	In this subsection, we will work in an arbitrary abelian group $G$.
	
	\begin{definition}
		We say a sequence $A=(a_1, a_2, ..., a_k)$ consisting of elements in $G$ is \textit{free of signed $t$-sums} if one cannot choose indices $1\le i_1<i_2<...<i_t\le k$ and signs $\varepsilon_1, ..., \varepsilon_t\in \{\pm1\}$ such that $\sum_{j=1}^t \eps_j i_j=0$. A sequence is \textit{free of signed $<t$-sums} if it is free of signed $t'$-sums for all $1\le t'<t$.
	\end{definition}
	
	\begin{remark}
		Note that in the special case $G=\ff_2^n$, a $k$-element sequence $A$ is free of signed $<k$-sums if and only if the ground set is a \textit{circuit} in the (linear) matroid on $[k]$ with independence defined by linear independence of the corresponding elements of $A$ over $\ff_2$. In particular, if the group has exponent $2$ then we can take all signs $\eps_j$ to be $+1$, i.e. instead of signed sums we just have sums.
		
		Only when $G=\ff_2^n$, by an abuse of terminology we will say that $A$ is a \textit{circuit} in this case. A multiset consisting of elements in $\ff_2^n$ is a \textit{circuit} if its elements form a circuit when listed in any order.
	\end{remark}
	
	We will now define the notion of a \textit{simple ordering} as in \cite{Costa_Morini_Pasotti_Pellegrini}.
	
	\begin{definition}
		We say a sequence $A=(a_1,a_2,...,a_k)$ consisting of elements in $G$ is in a \textit{simple ordering} if the partial sums $a_1, a_1+a_2, ..., a_1+a_2+...+a_k$ are pairwise distinct. A sequence or a multiset is \textit{simply orderable} if its elements can be arranged into a simple ordering.
	\end{definition}

	\begin{remark}
		Note that a zero-sum multiset of $k$ elements in $G$ is simply orderable if and only if its elements can be placed in a cyclic order such that any sequence of $k'$ elements along the cycle (with $1\le k'<k$) have a nonzero sum.
	\end{remark}

	In the literature, if the set of all elements of $G$ is simply orderable then $G$ is called \textit{sequenceable} \cite{Gordon}. In \cite{Costa_Morini_Pasotti_Pellegrini}, several conjectures on simple orderings of sets in cyclic, and more generally, abelian groups are discussed. We highlight the following conjecture:
	
	\begin{conjecture}[Costa, Morini, Pasotti, Pellegrini]\label{simply_orderable_conj}
		Let $G$ be an abelian group and $X\subseteq G\setminus \{0\}$. If $X$ is finite and zero-sum, and there is no element $x$ such that $\{x,-x\}\subseteq X$, then $X$ is simply orderable.
	\end{conjecture}

	In \cite{Costa_Morini_Pasotti_Pellegrini}, the authors prove that the conjecture holds for $|X|<10$, and by computer verification, they report that it is also true when $|G|\le 27$.

	We will now prove a technical proposition about merging together two circuits in an abelian group when certain conditions hold. 
	
	\begin{proposition}\label{merging_circuits}
		Let $G$ be an abelian group, and let $A=(a_1,a_2,...,a_k)$ and $B=(b_1,b_2,...,b_{\ell})$ be zero-sum sequences of elements in $G$, where $2\le k\le \ell$ are integers. Suppose that $A.B$ is free of signed $<k$-sums and $B$ is free of signed $<\ell$-sums. Then $A.B$ is simply orderable, except in the following two special cases:
		
		\begin{itemize}
			\item $k=\ell=2$, $A\sim B\sim (a,-a)$ for some $0\ne a\in G$,
			
			\item $k=2$, $\ell=3$, $A\sim (a,a)$ and $B\sim (a,d,a-d)$ for some $a,d\in G$ with $a,d\ne 0$, $2a=0$, $d\ne a$ and $a+2d\ne 0$.
		\end{itemize}
	\end{proposition}
	
	Note that the case $k\ge 3$ would follow from Conjecture \ref{simply_orderable_conj} with the set $X$ consisting of the elements of $A.B$. The proof of this proposition will be probabilistic: we consider a uniformly random cyclic ordering of the elements of $A.B$, and compute upper bounds on the probabilities of each subinterval along the cycle summing to $0$. When $k$ and $\ell$ are both small, some further casework will be needed.
	
	\begin{proof}
		Take a uniformly random cyclic ordering of $A.B$. For each $0\le t\le k+\ell$, denote by $p_t$ the probability that for a uniformly random subset $U\subseteq [k+\ell]$ of size $t$, $(A.B)_U$ sums to $0$. Then clearly $p_t=\frac{N_t}{\binom{k+\ell}{t}}$, where $N_t$ is the number of subsets $U$ of $[k+\ell]$ with $|U|=t$ and $\sum (A.B)_U=0$.
		
		Then for each fixed subinterval of length $t$ along the cycle, the probability that it sums to $0$ is $p_t$, and for a fixed $t$ there are $k+\ell$ such subintervals.
		
		Observe that $\sum (A.B)_U=0$ if and only if $\sum (A.B)_{[k+\ell]\setminus U}=0$, since $\sum A.B=0$. Because of this, when $2\mid k+\ell$ and $t=\frac{k+\ell}{2}$, it is sufficient to only consider $\frac{k+\ell}{2}$ of the possible subintervals (one from each complement pair). Furthermore, since $A.B$ is free of signed $<k$-sums, we have $p_t=N_t=0$ for each $0<t<k$, and by the prior complementing argument, $p_t=N_t=0$ for each $\ell<t<k+\ell$ as well.
		
		Taking a union bound, if we can show that
		
		$$\frac{k+\ell}{2}\left(p_k+p_{k+1}+...+p_{\ell}\right)<1$$
		
		then this means there exists an ordering without any zero-sum subinterval. Note that since $p_t=p_{k+\ell-t}$ for all $k\le t\le \ell$, if we convert each term to a length between $k$ and $\left\lfloor \frac{k+\ell}{2}\right\rfloor$, the sum contains each subinterval length the correct number of times, including the case of $t=\frac{k+\ell}{2}$ when $2\mid k+\ell$. Let us denote $\alpha=\frac{k+\ell}{2}\left(p_k+p_{k+1}+...+p_{\ell}\right)$.
		\medskip
		
		Now we prove that for any $S\subseteq [k]$, there exist at most two subsets $T\subseteq [\ell]$ such that $A_S.B_T$ is zero-sum.
		
	    If $T_1$ and $T_2$ both work where $T_2\neq T_1$ and $T_2\neq [\ell]\setminus T_1$, then $\sum (A_S.B_{T_1})-\sum (A_S.B_{T_2})=\sum B_{T_1\setminus T_2}-\sum B_{T_2\setminus T_1}=0$, meaning that $B$ has a proper nonempty subsequence giving a 0 sum when appropriately signed, so $B$ is not free of signed $<\ell$-sums, contradicting the hypotheses of the Proposition.\medskip
		
	    This means that the total number of $U\subseteq [k+\ell]$ such that $(A.B)_U$ is zero-sum is at most $2^{k+1}$. We have $N_0=N_{k+\ell}=1$, and so $N_k+N_{k+1}+...+N_{\ell}\le 2^{k+1}-2$.
		
		This means that
		\begin{flalign*}
		\alpha &=\frac{k+\ell}{2} \left(\frac{N_k}{\binom{k+\ell}{k}}+\frac{N_{k+1}}{\binom{k+\ell}{k+1}}+...+\frac{N_{\ell}}{\binom{k+\ell}{\ell}}\right)&\\
		&\le \frac{k+\ell}{2}\cdot \frac{1}{\binom{k+\ell}{k}}(N_k+N_{k+1}+...+N_{\ell})\le\frac{(k+\ell)(2^{k+1}-2)}{2\binom{k+\ell}{k}}=\frac{(k+\ell)(2^k-1)}{\binom{k+\ell}{k}}.
		\end{flalign*}
		
		Denoting $f(k,\ell)=\frac{(k+\ell)(2^k-1)}{\binom{k+\ell}{k}}$, if we can show that $f(k,\ell)<1$ for a particular pair $(k,\ell)$ then we will have proved the Proposition for that pair $(k,\ell)$. By direct computation, we have $f(2,6)=\frac67$, $f(3,6)=\frac34$, $f(4,6)=\frac57$, $f(5,6)=\frac{31}{42}$ and $f(6,6)=\frac{9}{11}$ which are all less than 1.
		
		Observe that $\frac{f(k,\ell+1)}{f(k,\ell)}=\frac{k+\ell+1}{k+\ell}\cdot \frac{\binom{k+\ell}{k}}{\binom{k+\ell+1}{k}}=\frac{k+\ell+1}{k+\ell}\cdot \frac{\ell+1}{k+\ell+1}=\frac{1+\ell}{k+\ell}<1$. So if $f(k,\ell)<1$ then $f(k,\ell')<1$ for all $\ell'>\ell$ too. 
		
		Further, $\frac{f(k+1,k+1)}{f(k,k)}=\frac{2k+2}{2k}\cdot \frac{\binom{2k}{k}}{\binom{2k+2}{k+1}}\cdot \frac{2^{k+1}-1}{2^k-1}=\frac{k+1}{k}\cdot \frac{k+1}{4k+2}\cdot \left(2+\frac{1}{2^k-1}\right)$ where for $k\ge 6$, $\frac{(k+1)^2}{k(4k+2)}<\frac13$ and $2+\frac{1}{2^k-1}<3$, meaning that the expression is less than $1$, so by induction, $f(k,k)<1$ for all $k\ge 6$, and combining this with the previous observations, we are done with all cases when $\ell\ge 6$.
		
		Therefore only the cases $(k,\ell)$ with $2\le k\le \ell \le 5$ remain. The proof for these cases can be found in the Appendix.
	\end{proof}
	
	\subsection{$\boldsymbol{u}$-good multisets}
	
	Throughout this subsection, let $\mathbf{u}\in \ff_2^n$ be a fixed nonzero element. We will define $\mathbf{u}$-good multisets, which when found among the prescribed differences, can provide a useful building block for making a perfect matching in the case when many of the given differences are equal to $\mathbf{u}$.
	
	\begin{definition}
		A sequence $(\mathbf{d}_1, \mathbf{d}_2, ..., \mathbf{d}_t)$ of elements in $\ff_2^n$ is said to be in a \textit{mod $\boldsymbol{u}$ simple order} if in $\ff_2^n/\langle \mathbf{u}\rangle$, the sequence $(\mathbf{d}_1+\langle \mathbf{u}\rangle, \mathbf{d}_2+\langle \mathbf{u}\rangle, ..., \mathbf{d}_t+\langle \mathbf{u}\rangle)$ is simply ordered.
		
		A multiset $\{\mathbf{d}_1, \mathbf{d}_2, ..., \mathbf{d}_t\}$ of elements in $\ff_2^n$ is said to be a \textit{mod $\boldsymbol{u}$ circuit} if in the group $\ff_2^n/\langle \mathbf{u}\rangle$, the multiset $\{\mathbf{d}_1+\langle \mathbf{u}\rangle, \mathbf{d}_2+\langle \mathbf{u}\rangle, ..., \mathbf{d}_t+\langle \mathbf{u}\rangle\}$ is a circuit.
	\end{definition}
	
	\begin{definition}\label{def_ugood}
		A multiset $D$ of elements in $\ff_2^n$ will be called \textit{$\boldsymbol{u}$-good} if the following conditions hold:
		
		\begin{itemize}
			\item $\sum D=|D|\mathbf{u}$,
			\item $D$ can be ordered in a mod $\mathbf{u}$ simple way.
		\end{itemize}
	\end{definition}
	
	$\mathbf{u}$-good multisets have the following important placement property.
	
	\begin{proposition}\label{ugood_basicproperty}
		Let $D$ be a $\boldsymbol{u}$-good multiset of elements in $\ff_2^n$. Then in any linear subspace $U\le \ff_2^n$ containing $\boldsymbol{u}$ and all elements of $D$, one can find disjoint pairs $(\boldsymbol{a}_1, \boldsymbol{b}_1), ..., (\boldsymbol{a}_t, \boldsymbol{b}_t)$ such that $\boldsymbol{a}_i-\boldsymbol{b}_i=\boldsymbol{d}_i$ for all $1\le i\le t$, and furthermore the union of the $t$ pairs is equal to the union of $t$ distinct $\langle \boldsymbol{u}\rangle$-cosets of $U$.
	\end{proposition}
	\begin{proof}
		Suppose that $(\mathbf{d}_1, ..., \mathbf{d}_t)$ is a mod $\mathbf{u}$ simple ordering of $D$. Consider the following pairs of elements in $U$: $(\mathbf{0}, \mathbf{d}_1), (\mathbf{d}_1+\mathbf{u}, \mathbf{d}_1+\mathbf{d}_2+\mathbf{u}), (\mathbf{d}_1+\mathbf{d}_2, \mathbf{d}_1+\mathbf{d}_2+\mathbf{d}_3), ..., (\mathbf{d}_1+\mathbf{d}_2+...+\mathbf{d}_{i-1}+(i-1)\mathbf{u}, \mathbf{d}_1+\mathbf{d}_2+...+\mathbf{d}_{i-1}+\mathbf{d}_i+(i-1)\mathbf{u}), ..., (\mathbf{d}_1+\mathbf{d}_2+...+\mathbf{d}_{t-1}+(t-1)\mathbf{u}, \mathbf{d}_1+\mathbf{d}_2+...+\mathbf{d}_{t-1}+\mathbf{d}_t+(t-1)\mathbf{u}=\mathbf{u})$. These pairs have difference $\mathbf{d}_1, \mathbf{d}_2, ..., \mathbf{d}_t$ respectively and are all disjoint since $\mathbf{0}, \mathbf{d}_1, \mathbf{d}_1+\mathbf{d}_2, ..., \mathbf{d}_1+\mathbf{d}_2+...+\mathbf{d}_{t-1}$ are pairwise distinct mod $\mathbf{u}$. Furthermore, these pairs fully cover $t$ distinct $\langle \mathbf{u}\rangle$-cosets in $U$.
	\end{proof}
	
	\begin{remark}
		Clearly Proposition \ref{ugood_basicproperty} also works for an affine subspace $\mathbf{x}+U$ instead of $U$, where $U$ contains $\mathbf{u}$ and all elements of $D$: just add $\mathbf{x}$ to $\mathbf{a}_i$ and $\mathbf{b}_i$ for all $1\le i\le t$.
	\end{remark}
	
	The following versatile proposition shows the usefulness of $\mathbf{u}$-good multisets:
	
	\begin{proposition}\label{ugood_partition}
		In $\ff_2^n$, let $S$ be a multiset with $\sum S=|S|\boldsymbol{u}$ which does not contain any elements congruent to $\boldsymbol{0}$ mod $\boldsymbol{u}$. Then the elements of $S$ can be partitioned into $\boldsymbol{u}$-good multisets.
	\end{proposition}
	\begin{proof}
		We will say that a mod $\mathbf{u}$ circuit $C$ has \textit{good parity} if $\sum C=|C|\mathbf{u}$, and \textit{bad parity} if $\sum C=(|C|+1)\mathbf{u}$.
		
		Iterate the following procedure: choose the smallest available mod $\mathbf{u}$ circuit among the elements of $S$. If there are multiple smallest circuits, then they can be chosen in any order, except that among the $2$-element circuits, we first choose the good-parity ones (those of the form $(\mathbf{s},\mathbf{s})$) and then the bad-parity ones (those of the form $(\mathbf{s}, \mathbf{s}+\mathbf{u})$).
		
		In this way, all elements of $S$ have been partitioned into mod $\mathbf{u}$ circuits, since $\sum S\equiv 0\pmod{\mathbf{u}}$ and each circuit also has $\mathbf{0}$ sum mod $\mathbf{u}$, so the remaining elements always have $\mathbf{0}$ sum mod $\mathbf{u}$, so are linearly dependent mod $\mathbf{u}$, so a mod $\mathbf{u}$ circuit can always be chosen as long as there are still some elements remaining.
		
		Furthermore, since $\sum S=|S|\mathbf{u}$, we have $\sum\limits_C (\sum C-|C|\mathbf{u})=\mathbf{0}$ (where $C$ iterates through the circuits in the obtained partitioning), so the total number of bad-parity circuits in the partitioning is even (as good-parity circuits contribute $\mathbf{0}$ to the sum and bad-parity circuits contribute $\mathbf{u}$).
		
		Now construct the classes of the partitioning in the following way: let each good circuit form its own class, and out of the bad circuits, form pairs in an arbitrary way, and every such pair of circuits is combined to form a class.
		
		We now show that every such class is $\mathbf{u}$-good indeed: the first condition is satisfied, as good-parity circuits have $\sum C=|C|\mathbf{u}$, and if $C_1$ and $C_2$ are bad-parity circuits, then $\sum (C_1.C_2)=\sum C_1+\sum C_2=(|C_1|+1)\mathbf{u}+(|C_2|+1)\mathbf{u}=|C_1.C_2|\mathbf{u}$.
		
		For a good-parity circuit $C=(\mathbf{c}_1, \mathbf{c}_2, ..., \mathbf{c}_t)$, every ordering is mod $\mathbf{u}$ simple: supposing we had $\mathbf{c}_1+...+\mathbf{c}_{\ell}\equiv \mathbf{c}_1+...+\mathbf{c}_{\ell'} \pmod{\mathbf{u}}$ for $0\le \ell<\ell'\le t-1$, then this would mean  $\mathbf{c}_{\ell+1}+\mathbf{c}_{\ell+2}+...+\mathbf{c}_{\ell'}\equiv \mathbf{0}\pmod{\mathbf{u}}$, meaning that $C$ has a linearly dependent proper subset, contradicting the fact that it is a circuit.
		
		For a pair of bad-parity circuits $A=(\mathbf{a}_1, ..., \mathbf{a}_k)$ and $B=(\mathbf{b}_1, ..., \mathbf{b}_{\ell})$, where $A$ was chosen earlier than $B$ without loss of generality, let $\overline{A}=(\overline{\mathbf{a}_1}, ..., \overline{\mathbf{a}_k})$ and $\overline{B}=(\overline{\mathbf{b}_1}, ..., \overline{\mathbf{b}_{\ell}})$ denote the sequences where each element is replaced by its $\langle \mathbf{u}\rangle$-coset. Now in the group $\ff_2^n/\langle \mathbf{u}\rangle$, apply Proposition \ref{merging_circuits} to $\overline{A}$ and $\overline{B}$, whose conditions are satisfied:
		
		\begin{itemize}
			\item they are zero-sum sequences with $2\le k\le \ell$,
			\item $\overline{A}.\overline{B}$ is free of $<k$-sums, as otherwise $A$ would not have been the smallest mod $\mathbf{u}$ circuit at the time it was chosen,
			\item $\overline{B}$ is free of $<\ell$-sums, as otherwise $B$ would not be a mod $\mathbf{u}$ circuit.
		\end{itemize}
		
		Also note that the two exceptional cases in this Proposition cannot occur, as both of them involve three equal elements in $\overline{A}.\overline{B}$ (in the first exceptional case we would have $a=-a$ since $\ff_2^n/\langle \mathbf{u}\rangle$ has exponent $2$), however any pairs of equal elements in $S$ have been removed at the beginning of the procedure (as a 2-element good-parity circuit), so the union of all bad circuits, and hence $\overline{A}.\overline{B}$, can only contain at most two copies of the same element mod $\mathbf{u}$.
		
		This means that by the Proposition, $\overline{A}.\overline{B}$ can be simply ordered, so $A.B$ can be simply ordered mod $\mathbf{u}$.
	\end{proof}
	
	\begin{remark}\label{ugood_partition_sizeremark}
		As in $\ff_2^n/\langle \mathbf{u}\rangle$, the largest independent set over $\ff_2$ has size $n-1$, each mod $\mathbf{u}$ circuit has size $\le n$. So the proof of Proposition \ref{ugood_partition} obtains a partitioning where each $\mathbf{u}$-good multiset has size $\le 2n$.
	\end{remark}
	
	\subsection{Proof of Theorem \ref{thm-n-2logn-1-osztaly}}
	
	We are now ready to prove our first main result:
	\medskip
	
	\noindent \textbf{Theorem \ref{thm-n-2logn-1-osztaly}.} \textit{The main conjecture is true in the case when the number of distinct values among the difference vectors is at most $n-2\log n-1$.}
	
	\begin{proof}
		Let the distinct values of the given differences be $\mathbf{u_1}$, $\mathbf{u_2}$, $\dots$, $\mathbf{u_t}$, where $t\le n-2\log n-1$, and for each $1\le i\le t$, $\mathbf{u_i}$ appears $n_i$ times with $n_1\ge n_2\ge \dots\ge n_t$. Here we have $\sum_{i=1}^t n_i=m=2^{n-1}$. In particular, $n_1\ge \frac{m}{t}$.
		
		In $\ff_2^n$, a vector space over $\ff_2$, let $U=\langle \mathbf{u_1},\dots, \mathbf{u_t}\rangle$ and let $k=\dim U$. (Then $k\le t$.) We can assume that $k\ge 2$, as otherwise $t=1$, and this is a case that we have already seen. Call the $U$-cosets of $\ff_2^n$ \textit{layers}. Then in the perfect matching that we will make, each vector pair has to be within a layer. We will create perfect matchings of each layer separately, and we will not modify any finished layers later. Altogether there are $2^{n-k}$ layers.
		\medskip
		
		\noindent Our algorithm will consist of 3 phases:
		
		\begin{itemize}
			\item \textit{Phase 1:} We will create perfect matchings in some (less than $t$) layers in such a way that an even number of vectors will remain in each difference class.
			
			\item \textit{Phase 2:} We will create perfect matchings in some (less than $t$) layers in such a way that in each difference class, the number of remaining vectors will be divisible by $2^{k-1}$.
			
			\item \textit{Phase 3:} All of the remaining differences will be used to create homogeneous layers (i.e. layers consisting of differences from only one class).\\
		\end{itemize}
		
		\noindent Now we explain the details of each phase.
		
		\noindent \textit{Phase 1.} Let $H=\{\mathbf{u_i}: 2\le i\le t, n_i\equiv 1\pmod{2}\}$, and we will use the notation $\mathbf{u}=\mathbf{u_1}$.
		
		Using Proposition \ref{ugood_partition}, we partition $H$ into $\mathbf{u}$-good multisets. This is possible, since $0=\sum_{i=1}^t n_i\mathbf{u}_i = n_1\mathbf{u}+\sum H$, and the number of elements in $H$ has the same parity as the total number of all differences not equal to $\mathbf{u}$, which in turn has the same parity as the total number of differences equal to $\mathbf{u}_1$, and hence $|H|\mathbf{u}=n_1\mathbf{u}=\sum H$.
		
		Take the $\mathbf{u}$-good multisets in turn, and for each multiset $D$, make a perfectly matched layer in the following way: use Proposition \ref{ugood_basicproperty} for an unused layer $\mathbf{x}+U$. It can be used since $U$ contains $\mathbf{u}$ and all elements of $D$. So we get $|D|$ disjoint pairs in $\mathbf{x}+U$ using the differences in $D$, whose union is equal to a union of $\langle \mathbf{u}\rangle$-cosets, therefore the matching of the layer can be completed using differences of value $\mathbf{u}$.
		
		After this phase, the number of remaining differences is even in the classes of $\mathbf{u}_2, ..., \mathbf{u}_t$, and this is true for $\mathbf{u}$ as well, as each remaining layer requires an even number of differences in total.\\
		
		\noindent \textit{Phase 2.} Perform the following step for each $2\le i\le t$ in turn.
		
		If the number of remaining copies of $\mathbf{u_i}$ has a remainder of $m_i$ modulo $2^{k-1}$ (where $2\le m_i\le 2^{k-1}-2$ is even), then let us make a perfect matching of a new layer using $m_i$ copies of $\mathbf{u_i}$ and $2^{k-1}-m_i$ copies of $\mathbf{u}$. This can be done, as this is the problem in the main conjecture for $\ff_2^k$ (actually for a translate of it, but this does not matter), and as $m_i$ and $2^{k-1}-m_i$ are even, the sum of the difference vectors we need to use is $\textbf{0}$, and we have already resolved the conjecture in the case of two difference classes.
		
		(If $m_i=0$ then we do not need to do anything with the $i$-th class.)\\
		
		\noindent \textit{Phase 3.} As the number of remaining vectors in each class is divisible by $2^{k-1}$, and for any $\mathbf{0}\ne \mathbf{d}\in U$ we can partition $U$ (or any translate of it) into $2^{k-1}$ pairs of difference $\mathbf{d}$, this phase can be trivially performed, completing the perfect matching of $\ff_2^n$ in the required manner.\\
		
		Observe that all three phases can always be performed: in phase 1, in every layer we use at least 3 elements of $H$, so we make at most $\frac{t-1}{3}$ completed layers. And in phase 2, we make less than $t$ layers. So in the first two phases, altogether we used $\le \frac{4}{3}(t-1)\cdot 2^{k-1}$ copies of $\mathbf{u}$. And indeed we did have this many copies of $\mathbf{u}$ at our disposal, as
		$$\frac43t(t-1)\cdot 2^{k-1}\le \frac43t(t-1)\cdot 2^{t-1}\le \frac43t(t-1)\cdot 2^{n-2\log n-2}=\frac13 t(t-1)\frac{2^n}{n^2}<\frac13\cdot 2^n<2^{n-1}$$
		
		and so $\frac43(t-1)\cdot 2^{k-1}< \frac{2^{n-1}}{t}=\frac{m}{t}\le n_1$.
	\end{proof}
	
	\section{Perfect matching in the case of many equal vectors}
	
	In this chapter, we resolve the main conjecture (for sufficiently large $n$) in the special case when at least a fraction $\frac12+\eps$ of the difference vectors are all equal, and the others are arbitrary. So in contrast to the theorem of Balister, Győri and Schelp (see Theorem \ref{feleazonos_tobbiparokban}), here we do not require that all differences appear an even number of times.
	
	\subsection{Affine flats}
	
	In order to create a perfect matching of $\ff_2^n$ according to the given differences in the main conjecture, it is worthwhile to group the given differences into classes with nice additive combinatorial properties. An example of such a class would be a $k$-dimensional affine subspace of $\ff_2^n$ (also called a \textit{$k$-flat}).
	
	For a fixed $k\ge 1$, we are interested in the minimal integer $s$ (as a function of $n$) such that any subset $S\subseteq \ff_2^n$ of size $s$ is guaranteed to contain a $k$-flat.
	
	In the case $k=2$, a set $S\subseteq \ff_2^n$ contains a $2$-flat if and only if it is not a Sidon set. (A subset $S$ of an abelian group $G$ is called a \textit{Sidon set} if the only solutions to $a+b=c+d$ with $a,b,c,d\in G$ are the trivial solutions where $(c,d)$ is a permutation of $(a,b)$.) Bose and Ray-Chaudhuri \cite{BR60} have shown that for $k=2$, the smallest such $s$ is $\Theta(2^{n/2})$.
	
	It turns out that even for larger (fixed) values of $k$, the smallest size $s$ required is still exponentially small compared to $2^n$. Going along the lines of Szemerédi's Cube Lemma (see \cite[Corollary 2.1]{Setyawan}), a result of Bonin and Qin \cite[Lemma 21]{Bonin_Qin} states the following:
	
	\begin{proposition}\label{kflat_findable}
		Given integers $n\ge k\ge 2$, and any $s\ge 2\cdot 2^{n\left(1-\frac{1}{2^{k-1}}\right)}$, any subset $S\subseteq \ff_2^n$ with $|S|\ge s$ contains a $k$-flat.
	\end{proposition}
	
	This will mean that for $n$ sufficiently large compared to $k$, and a given set $D$ of differences with $|D|\gg 2^{n\left(1-\frac{1}{2^{k-1}}\right)}$, almost all elements of $D$ may be partitioned into $k$-flats.
	
	Note that an easy corollary of Lemma \ref{vektorterben_osszeg0} is that for any $k\ge 2$, the sum of the elements of any $k$-flat in $\ff_2^n$ is $\mathbf{0}$. We will use this fact later.
	
	\subsection{Blocking ratios}
	
	In order for us to flexibly find partial matchings among a certain set of elements in $\ff_2^n$ according to given sequences of differences, we define the following notions:
	
	\begin{definition}
		Given a multiset $D=\{\mathbf{d}_1, \mathbf{d}_2, ..., \mathbf{d}_k\}$ of differences in $\ff_2^n$, define a subset $S\subseteq \ff_2^n$ to be \textit{blocking for $D$} if in $\ff_2^n\setminus S$ there do not exist $k$ pairwise disjoint pairs $(\mathbf{a}_1, \mathbf{b}_1), (\mathbf{a}_2,\mathbf{b}_2), ..., (\mathbf{a}_k,\mathbf{b}_k)$ of elements such that $\mathbf{a}_i-\mathbf{b}_i=\mathbf{d}_i$ for all $1\le i\le k$.
		
		Let the \textit{blocking ratio} of $D$ be defined as 
		
		$$\beta(D)=\frac{\min \{ |S|:~S\subseteq \ff_2^n,~S\textrm{ is blocking for }D\}}{2^n}.$$
	\end{definition}
	
	\vspace{5mm}
	
	For example, for any nonzero $\mathbf{d}\in \ff_2^n$, $\{\mathbf{d}\}$ has a blocking ratio of $\frac12$, since a set $S\subseteq \ff_2^n$ is blocking for $\{\mathbf{d}\}$ if and only if it contains at least one element of each $\langle \mathbf{d}\rangle$-coset.
	
	Since in our algorithm, we will pack differences while respecting cosets of $\langle \mathbf{u}\rangle$, where $\mathbf{u}$ is the most commonly appearing difference, it is worthwhile to also define the following analogous version of the blocking ratio:
	
	\begin{definition}
		Let $\mathbf{u}\in \ff_2^n\setminus \{\mathbf{0}\}$ be fixed. Define a subset $S\subseteq \ff_2^n$ to be \textit{$\boldsymbol{u}$-blocking for $D$} if $S$ is a union of some $\langle \mathbf{u}\rangle$-cosets and in $\ff_2^n\setminus S$ there do not exist $k$ pairwise disjoint pairs $(\mathbf{a}_1, \mathbf{b}_1), (\mathbf{a}_2,\mathbf{b}_2), ..., (\mathbf{a}_k,\mathbf{b}_k)$ of elements such that the following two conditions hold:
		\begin{itemize}
			\item $\mathbf{a}_i-\mathbf{b}_i=\mathbf{d}_i$ for all $1\le i\le k$, 
			\item and $\{\mathbf{a}_1, ..., \mathbf{a}_k, \mathbf{b}_1, ..., \mathbf{b}_k\}$ is equal to a union of some $\langle \mathbf{u}\rangle$-cosets.
		\end{itemize}
		
		Let the \textit{$\boldsymbol{u}$-blocking ratio} of $D$ be defined as
		
		$$\beta_{\mathbf{u}}(D)=\frac{\min \{ |S|:~S\subseteq \ff_2^n,~S\textrm{ is }\mathbf{u}\textrm{-blocking for }D\}}{2^n}.$$
	\end{definition}
	
	\vspace{5mm}
	
	As an example, $\beta_{\mathbf{u}}(\{\mathbf{u}\})=1$, and for any $\mathbf{0}\ne \mathbf{d}\ne \mathbf{u}$, $\beta_{\mathbf{u}}(\{\mathbf{d},\mathbf{d}\})=\frac12$, since in $\ff_2^n/\langle \mathbf{u}\rangle$, the cosets can be paired to each other based on congruence modulo $\mathbf{d}$, and a set (consisting of full $\langle \mathbf{u}\rangle$-cosets) is $\mathbf{u}$-blocking for $\{\mathbf{d},\mathbf{d}\}$ if and only if at least one $\langle \mathbf{u}\rangle$-coset is chosen from each such pair. If assigning disjoint pairs to the differences in $D$ is impossible in such a way that the pairs together only cover full $\langle \mathbf{u}\rangle$-cosets (such as in the case $D=\{\mathbf{d}\}$ for $\mathbf{0}\ne \mathbf{d}\ne \mathbf{u}$), then $S=\emptyset$ will be $\mathbf{u}$-blocking for $D$, so $\beta_{\mathbf{u}}(D)=0$.
	
	The following proposition will be used for constructing a perfect matching of $\ff_2^n$ according to a list of given differences grouped into classes that have large enough $\mathbf{u}$-blocking ratios:
	
	\begin{proposition}\label{construction_scheme}
		In the main conjecture, suppose that the multiset of given differences $(\boldsymbol{d}_1, \boldsymbol{d}_2, ..., \boldsymbol{d}_m)$ is partitioned into several classes $D_1, D_2, ..., D_{\ell}$ such that for some fixed value of $\boldsymbol{u}\in \ff_2^n\setminus \{\boldsymbol{0}\}$, the following holds for all $1\le i\le \ell$: $$\beta_{\mathbf{u}}(D_i)> \frac{\sum\limits_{j=1}^{i-1} |D_j|}{2^{n-1}}.$$
		
		Then there is an appropriate perfect matching of $\ff_2^n$ corresponding to the given differences.
	\end{proposition}
	\begin{proof}
		We construct a perfect matching by iteratively going through each class $D_i$ in increasing order of $i$. For each $D_i$, we assign to the differences in $D_i$ disjoint pairs of elements of $\ff_2^n$ that have not yet been used for previous classes  $D_j$ ($j<i$). Our procedure will respect $\langle \mathbf{u}\rangle$-cosets, meaning that for each $i$, the elements of $\ff_2^n$ assigned to the differences in $D_i$ will form a union of $\langle \mathbf{u}\rangle$-cosets.
		
		For each $D_i$, the number of elements of $\ff_2^n$ already used for previous classes is $2\sum_{j=1}^{i-1} |D_j|$, and since $$\beta_{\mathbf{u}}(D_i)>\frac{2\sum\limits_{j=1}^{i-1} |D_j|}{2^n},$$
		
		the already-used elements are not $\mathbf{u}$-blocking for $D_i$, hence it is possible to assign pairs of new elements to $D_i$ in a way that respects $\langle \mathbf{u}\rangle$-cosets.
	\end{proof}
	
	\subsection{Blocking ratio of $\boldsymbol{u}$-good multisets}
	
	In this subsection and the following one, the image of a vector $\mathbf{v}\in \ff_2^n$ under the quotient map $q: \ff_2^n\to \ff_2^n/\langle \mathbf{u}\rangle$ will be denoted by $\overline{\mathbf{v}}$. When given a subset $S$ of $\ff_2^n$ which is a union of $\langle \mathbf{u}\rangle$-cosets, define $\overline{S}=q(S)\subseteq \ff_2^n$. Note that this correspondence can be used in both directions, since $S$ can be recovered from $\overline{S}$ by taking $S=q^{-1}(\overline{S})$.
	
	We have seen that $\mathbf{u}$-good multisets (as per Definition \ref{def_ugood}) have a positive $\mathbf{u}$-blocking ratio: this is equivalent to the statement of Proposition \ref{ugood_basicproperty}. Using the following lemma, we can place a better lower bound on $\beta_{\mathbf{u}}(D)$ for any $\mathbf{u}$-good multiset $D$.
	
	Recall that for an abelian group $G$, if $X\subseteq G$ then subsets of $G$ of the form $g+X=\{g+x: x\in X\}$, for arbitrary $g\in G$, are called \textit{translates} of $X$.
	
	\begin{lemma}\label{diszjunktmintak} Let $G$ be a finite abelian group, and let $X\subseteq G$. Then in $G$, we can select at least $\frac{|G|}{|X|(|X|-1)+1}$ pairwise disjoint translates of $X$.
	\end{lemma}
	
	\begin{proof}
		Keep choosing translates of $X$ greedily which are disjoint from the previously chosen ones. Let the chosen translates be $g_1+X$, $g_2+X$, ..., $g_\ell+X$, and suppose that no further translate of $X$ can be chosen that is disjoint from these.
		
		Then for every $g_{\ell+1}\in G$, $g_{\ell+1}+X$ intersects at least one of the earlier translates, so there exist  $x,x'\in X$ and $1\le i\le \ell$ such that $g_{\ell+1}+x=g_i+x'$, so $g_{\ell+1}=g_i+x'-x$. But an upper bound on the number of elements expressible in this form is $\ell(1+|X|(|X|-1))$ (as $x'-x$ can take at most $|X|(|X|-1)$ nonzero values), so $|G|\le \ell(1+|X|(|X|-1))$, proving the lemma.
	\end{proof}
	
	\begin{remark}\label{diszjunktmintak_2exp}
		If the group $G$ has exponent 2 (that is, for every $g\in G$ we have $g+g=0$), then $x'-x$ can take at most $\binom{|X|}{2}$ nonzero values, hence in this case, the lemma can be improved to say that at least $\frac{|G|}{\binom{|X|}{2}+1}$ pairwise disjoint translates of $X$ can be selected.
	\end{remark}
	
	\begin{proposition}\label{beta_greedybound}
		Let $D=\{\boldsymbol{d}_1, \boldsymbol{d}_2, ..., \boldsymbol{d}_t\}$ be a $\boldsymbol{u}$-good multiset for some nonzero $\boldsymbol{u}\in \ff_2^n$. Then $\beta_{\boldsymbol{u}}(D)\ge \frac{1}{\binom{t}{2}+1}$.
	\end{proposition}
	
	\begin{proof}
		Suppose on the contrary that $S\subseteq \ff_2^n$ is $\mathbf{u}$-blocking for $D$ with $|S|<\frac{2^n}{\binom{t}{2}+1}$. Assuming $(\overline{\mathbf{d}_1}, ..., \overline{\mathbf{d}_t})$ is simply ordered in $\ff_2^n/\langle \mathbf{u}\rangle$, the partial sums $\mathbf{d}_1$, $\mathbf{d}_1+\mathbf{d}_2$, ..., $\mathbf{d}_1+\mathbf{d}_2+...+\mathbf{d}_t$ are pairwise distinct modulo $\mathbf{u}$. Consider the following set $\overline{X}\subseteq \ff_2^n/\langle \mathbf{u}\rangle$: $\overline{X}=\{\overline{\mathbf{d}_1}, \overline{\mathbf{d}_1}+\overline{\mathbf{d}_2}, ..., \overline{\mathbf{d}_1}+\overline{\mathbf{d}_2}+...+\overline{\mathbf{d}_t}=\overline{\mathbf{0}}\}$ (where we have used the fact that $\mathbf{d}_1+\mathbf{d}_2+...+\mathbf{d}_t=t\mathbf{u}\equiv \mathbf{0}\pmod{\mathbf{u}}$).
		
		Now $|\overline{S}|<\frac{|\ff_2^n/\langle \mathbf{u}\rangle|}{\binom{t}{2}+1}$. Now by Remark \ref{diszjunktmintak_2exp}, we can select at least $\frac{|\ff_2^n/\langle \mathbf{u}\rangle|}{\binom{t}{2}+1}$ pairwise disjoint translates of $\overline{X}$ in $\ff_2^n/\langle \mathbf{u}\rangle$. At least one of these translates must be disjoint from $\overline{S}$; let us choose such a translate $\overline{Y}$, which is equal to $\{\overline{\mathbf{y}\vphantom{d}}, \overline{\mathbf{y}\vphantom{d}}+\overline{\mathbf{d}_1}, ..., \overline{\mathbf{y}\vphantom{d}}+\overline{\mathbf{d}_1}+...+\overline{\mathbf{d}_{t-1}}\}$ for some $\mathbf{y}\in \ff_2^n$. Then $Y$ is disjoint from $S$, is a union of $\langle \mathbf{u}\rangle$-cosets, and can be subdivided into the $t$ pairs $(\mathbf{y}, \mathbf{y}+\mathbf{d}_1)$, $(\mathbf{y}+\mathbf{d}_1+\mathbf{u}, \mathbf{y}+\mathbf{d}_1+\mathbf{d}_2+\mathbf{u})$, $(\mathbf{y}+\mathbf{d}_1+\mathbf{d}_2, \mathbf{y}+\mathbf{d}_1+\mathbf{d}_2+\mathbf{d}_3)$, ..., $(\mathbf{y}+\mathbf{d}_1+\mathbf{d}_2+...+\mathbf{d}_{t-1}+(t-1)\mathbf{u}, \mathbf{y}+\mathbf{d}_1+\mathbf{d}_2+...+\mathbf{d}_{t-1}+\mathbf{d}_t+(t-1)\mathbf{u}=\mathbf{y}+\mathbf{u})$ with differences $\mathbf{d}_1$, $\mathbf{d}_2$, ..., $\mathbf{d}_t$ respectively. So $S$ is not actually $\mathbf{u}$-blocking for $D$, which is a contradiction.
	\end{proof}
	
	\subsection{Blocking ratio of $k$-flats}
	
	\begin{definition}
		A $k$-flat $F\subseteq \ff_2^n$ will be called \textit{$\boldsymbol{u}$-nice} if no element of $F$ is congruent to $\mathbf{0}$ mod $\mathbf{u}$ and no two elements of $F$ are congruent to each other mod $\mathbf{u}$.
	\end{definition}
	
	Observe that any $2$-flat $F\subseteq \ff_2^n$ can be written in the form $F=\{\mathbf{x}, \mathbf{x}+\mathbf{a}_1, \mathbf{x}+\mathbf{a}_2, \mathbf{x}+\mathbf{a}_1+\mathbf{a}_2\}$ for some $\mathbf{x},\mathbf{a}_1,\mathbf{a}_2\in \ff_2^n$ and $\mathbf{a}_1\ne \mathbf{a}_2$, $\mathbf{a}_1,\mathbf{a}_2\ne \mathbf{0}$. It is easy to check that if $F$ is $\mathbf{u}$-nice, then $\{\mathbf{x}, \mathbf{x}+\mathbf{a}_1, \mathbf{x}+\mathbf{a}_2, \mathbf{x}+\mathbf{a}_1+\mathbf{a}_2\}$ is a $\mathbf{u}$-good multiset and hence by Proposition \ref{beta_greedybound}, $\beta_{\mathbf{u}}(F)\ge \frac17$. Let us improve this result:
	
	\begin{proposition}\label{blockingratio_2flat}
		For a $\mathbf{u}$-nice $2$-flat $F\subseteq \ff_2^n$, we have $\beta_{\mathbf{u}}(F)\ge \frac38$.
	\end{proposition}
	\begin{proof}
		Write $F=\mathbf{x}+L$ for $L=\langle \mathbf{a}_1,\mathbf{a}_2\rangle$. Since $F$ is $\mathbf{u}$-nice, the elements $\mathbf{x},\mathbf{a}_1,\mathbf{a}_2$ are linearly independent mod $\mathbf{u}$. (Otherwise one could find an element of $F$ congruent to $\mathbf{0}$ or two elements of $F$ congruent to each other mod $\mathbf{u}$.) So let $C=\langle \mathbf{u},\mathbf{x},\mathbf{a}_1,\mathbf{a}_2\rangle$ which is a $4$-dimensional subspace of $\ff_2^n$. Also by the $\mathbf{u}$-nice condition, the following three orderings of $F$ are all simple mod $\mathbf{u}$:
		\begin{itemize}
			\item $F_1=(\mathbf{x}+\mathbf{a}_1, \mathbf{x}+\mathbf{a}_2, \mathbf{x}+\mathbf{a}_1+\mathbf{a}_2, \mathbf{x})$,
			\item $F_2=(\mathbf{x}+\mathbf{a}_1, \mathbf{x}+\mathbf{a}_1+\mathbf{a}_2, \mathbf{x}+\mathbf{a}_2, \mathbf{x})$,
			\item $F_3=(\mathbf{x}+\mathbf{a}_2, \mathbf{x}+\mathbf{a}_1, \mathbf{x}+\mathbf{a}_1+\mathbf{a}_2, \mathbf{x})$.
		\end{itemize}
		
		Let $V=\ff_2^n$, and then $\overline{V}=\ff_2^n/\langle \mathbf{u}\rangle$. In $\overline{V}$, take the linear subspace $\overline{C}=\langle \overline{\mathbf{x}},\overline{\mathbf{a}_1},\overline{\mathbf{a}_2}\rangle$ and consider each $\overline{C}$-coset of $\overline{V}$. Suppose that $S$ is $\mathbf{u}$-blocking for $F$, and there exists a $\overline{C}$-coset $\overline{\mathbf{y}}+\overline{C}$ of $\overline{V}$ such that at most two of the eight $\langle \mathbf{u}\rangle$-cosets in $\overline{\mathbf{y}}+\overline{C}$ are contained in $\overline{S}$.
		
		According to the basic construction in Proposition \ref{ugood_basicproperty}, if $S$ is $\mathbf{u}$-blocking for a $\mathbf{u}$-good multiset simply ordered as $D=(\mathbf{d}_1, \mathbf{d}_2, ..., \mathbf{d}_t)$ then for every $\mathbf{x}\in \ff_2^n$, $S$ must contain at least one $\langle \mathbf{u}\rangle$-coset out of the partial sums $\overline{\mathbf{x}\vphantom{d}}, \overline{\mathbf{x}\vphantom{d}}+\overline{\mathbf{d}_1}, \overline{\mathbf{x}\vphantom{d}}+\overline{\mathbf{d}_1}+\overline{\mathbf{d}_2}, ..., \overline{\mathbf{x}\vphantom{d}}+\overline{\mathbf{d}_1}+\overline{\mathbf{d}_2}+...+\overline{\mathbf{d}_{t-1}}$.
		
		Applying this to $F_1$, $F_2$ and $F_3$, this means that for each $\overline{\mathbf{z}}\in \overline{\mathbf{y}}+\overline{C}$, $\overline{S}$ contains at least one element from each $\overline{\mathbf{z}}+\overline{R_j}$ ($j=1,2,3$) for the following sets $\overline{R_j}$:
		\begin{itemize}
			\item $\overline{R_1}=\{\overline{\mathbf{0}}, \overline{\mathbf{x}\vphantom{0}}+\overline{\mathbf{a}_1\vphantom{0}}, \overline{\mathbf{a}_1\vphantom{0}}+\overline{\mathbf{a}_2\vphantom{0}}, \overline{\mathbf{x}\vphantom{0}}\}$,
			\item $\overline{R_2}=\{\overline{\mathbf{0}}, \overline{\mathbf{x}\vphantom{0}}+\overline{\mathbf{a}_1\vphantom{0}}, \overline{\mathbf{a}_2\vphantom{0}}, \overline{\mathbf{x}\vphantom{0}}\}$,   
			\item $\overline{R_3}=\{\overline{\mathbf{0}}, \overline{\mathbf{x}\vphantom{0}}+\overline{\mathbf{a}_2\vphantom{0}}, \overline{\mathbf{a}_1\vphantom{0}}+\overline{\mathbf{a}_2\vphantom{0}}, \overline{\mathbf{x}\vphantom{0}}\}$.
		\end{itemize}
		
		Without loss of generality, assume $\overline{\mathbf{y}}\in \overline{S}$. Then depending on whether $\overline{\mathbf{y}}+\overline{C}$ has another element in $\overline{S}$, and what that element is, we can always find $\overline{\mathbf{z}}\in \overline{\mathbf{y}}+\overline{C}$ and $j=1,2,3$ with none of the elements of $\overline{\mathbf{z}}+\overline{R_j}$ contained in $\overline{S}$, causing a contradiction. For each case, the table below lists one such possible choice of $\overline{\mathbf{z}}+\overline{R_j}$:
		
		\begin{center}
			\begin{tabular}{|c|c|}
				\hline
				\textbf{Elements of $(\overline{\mathbf{y}}+\overline{C})\cap \overline{S}$} & \textbf{$\overline{\mathbf{z}}+\overline{R_j}$ with contradiction} \\
				\hline
				$ \overline{\mathbf{y}}$ & $ \overline{\mathbf{y}}+\overline{\mathbf{x}}+\overline{\mathbf{a}_2}+\overline{R_1}$ \\
				\hline
				$ \overline{\mathbf{y}},  \overline{\mathbf{y}}+\overline{\mathbf{a}_1}$ & $ \overline{\mathbf{y}}+\overline{\mathbf{x}}+\overline{\mathbf{a}_2}+\overline{R_1}$ \\
				\hline
				$ \overline{\mathbf{y}},  \overline{\mathbf{y}}+\overline{\mathbf{a}_2}$ & $ \overline{\mathbf{y}}+\overline{\mathbf{a}_1}+\overline{\mathbf{a}_2}+\overline{R_2}$ \\
				\hline
				$ \overline{\mathbf{y}},  \overline{\mathbf{y}}+\overline{\mathbf{a}_1}+\overline{\mathbf{a}_2}$ & $ \overline{\mathbf{y}}+\overline{\mathbf{a}_2}+\overline{R_1}$ \\
				\hline
				$ \overline{\mathbf{y}},  \overline{\mathbf{y}}+\overline{\mathbf{x}}$ & $\overline{\mathbf{y}}+\overline{\mathbf{x}}+\overline{\mathbf{a}_2}+\overline{R_1}$ \\
				\hline
				$ \overline{\mathbf{y}},  \overline{\mathbf{y}}+\overline{\mathbf{x}}+\overline{\mathbf{a}_1}$ & $ \overline{\mathbf{y}}+\overline{\mathbf{x}}+\overline{\mathbf{a}_1}+\overline{\mathbf{a}_2}+\overline{R_1}$ \\
				\hline
				$ \overline{\mathbf{y}},  \overline{\mathbf{y}}+\overline{\mathbf{x}}+\overline{\mathbf{a}_2}$ & $ \overline{\mathbf{y}}+\overline{\mathbf{x}}+\overline{\mathbf{a}_1}+\overline{\mathbf{a}_2}+\overline{R_1}$ \\
				\hline
				$ \overline{\mathbf{y}},  \overline{\mathbf{y}}+\overline{\mathbf{x}}+\overline{\mathbf{a}_1}+\overline{\mathbf{a}_2}$ & $ \overline{\mathbf{y}}+\overline{\mathbf{x}}+\overline{\mathbf{a}_2}+\overline{R_1}$ \\
				\hline
			\end{tabular}
		\end{center}
		
		This means that in every $\overline{C}$-coset of $\overline{V}$, at least 3 of the 8 elements are contained in $\overline{S}$. Therefore, $\frac{|S|}{2^n}\ge \frac38$.
	\end{proof}
	
	\begin{remark}\label{blockingratio_2flat_eqremark}
		In the proof of the previous proposition, if $\overline{S}$ contains three $\langle \mathbf{u}\rangle$-cosets of $\overline{\mathbf{y}}+\overline{C}$ then a case-by-case check reveals that we can always find $\overline{\mathbf{z}}+\overline{R_j}$ disjoint from $\overline{S}$ (where $\overline{\mathbf{z}}\in \overline{\mathbf{y}}+\overline{C}$), except in the case when the three $\langle \mathbf{u}\rangle$-cosets are in the same $\langle \overline{\mathbf{a}_1},\overline{\mathbf{a}_2}\rangle$-coset.
	\end{remark}
	
	\begin{proposition}\label{blockingratio_kflat}
		Let $k,n\ge 2$. For a $\boldsymbol{u}$-nice $k$-flat $F\subseteq \ff_2^n$, we have 
		
		$$\beta_{\boldsymbol{u}}(F)\ge \frac12-\frac{1}{2^{k+1}}-\frac{1}{2^{n-k(2^k+2)}}.$$
	\end{proposition}
	
	\begin{proof}
		Note that Proposition \ref{blockingratio_2flat} already proves the case $k=2$. 
		
		Letting $\mathbf{x}$ be an arbitrary element of $F$, we can write $F=\mathbf{x}+L$ where $L=\langle \mathbf{a}_1,\mathbf{a}_2,...,\mathbf{a}_k\rangle$ with $\mathbf{a}_1,...,\mathbf{a}_k\in \ff_2^n$ linearly independent. Again we use the notation $V=\ff_2^n$, and then  $\overline{V}=\ff_2^n/\langle \mathbf{u}\rangle$. Since $F$ is $\mathbf{u}$-nice, the elements $\mathbf{x},\mathbf{a}_1,...,\mathbf{a}_k$ are linearly independent mod $\mathbf{u}$, so let us define $C=\langle \mathbf{u},\mathbf{x},\mathbf{a}_1,...,\mathbf{a}_k\rangle=\langle \mathbf{u},\mathbf{x}\rangle \oplus L$ which is a $(k+2)$-dimensional subspace of $V$, and correspondingly $\overline{C}=\langle \overline{\mathbf{x}},\overline{\mathbf{a}_1},...,\overline{\mathbf{a}_k}\rangle$, a $(k+1)$-dimensional subspace of $\overline{V}$ satisfying $\overline{C}=\langle \overline{\mathbf{x}}\rangle \oplus \overline{L}$.
		
		Given $H\subseteq V$ which is a union of $\langle \mathbf{u}\rangle$-cosets, and a $\mathbf{u}$-nice $2$-flat $E$, let us say that $\overline{H}$ \textit{accommodates} $E$ if there exists $\overline{\mathbf{z}}\in \overline{V}$ and an ordering $(\mathbf{e}_1,\mathbf{e}_2,\mathbf{e}_3,\mathbf{e}_4)$ of $E$ such that $\{\overline{\mathbf{z}}, \overline{\mathbf{z}}+\overline{\mathbf{e}_1}, \overline{\mathbf{z}}+\overline{\mathbf{e}_1}+\overline{\mathbf{e}_2}, \overline{\mathbf{z}}+\overline{\mathbf{e}_1}+\overline{\mathbf{e}_2}+\overline{\mathbf{e}_3}\}\subseteq\overline{H}$. If this is true, then (by the idea of Proposition \ref{ugood_basicproperty}, as $E$ is $\mathbf{u}$-good with all orderings simple mod $\mathbf{u}$) we can select four pairwise disjoint pairs of elements in $H$ with differences $\mathbf{e}_1,\mathbf{e}_2,\mathbf{e}_3,\mathbf{e}_4$, such that their union is a union of $\langle \mathbf{u}\rangle$-cosets.
		
		Consider the $\overline{C}$-cosets of $\overline{V}$. We would like to partition $F$ into $2$-flats $E_1, E_2, ..., E_{2^{k-2}}$ and find distinct $\overline{C}$-cosets $\overline{\mathbf{y}_i}+\overline{C}$ ($1\le i\le 2^{k-2}$) in $\overline{V}$ such that for every $i$, $(\overline{\mathbf{y}_i}+\overline{C})\setminus \overline{S}$ accommodates $E_i$. If this can be done, then $S$ is not $\mathbf{u}$-blocking for $F$, as we can combine the element pairs obtained for each $E_i$ to get $2^k$ pairwise disjoint pairs in $\ff_2^n\setminus S$ such that each element $\mathbf{f}$ of $F$ has one pair with difference $\mathbf{f}$, and the pairs altogether form a union of $\langle \mathbf{u}\rangle$-cosets.
		
		We start with the following auxiliary claim:
        \medskip
		
		\noindent \textit{Claim 1.} Let $k\ge 2$ and $F\subseteq \ff_2^n$ be a $\mathbf{u}$-nice $k$-flat. Define $L$ and $C$ as above. Then if $\overline{H}\subseteq \overline{C}$ is a subset having at least 2 elements in one of the two $\overline{L}$-cosets of $\overline{C}$ and at least $2^k-1$ elements in the other $\overline{L}$-coset, then $F$ can be partitioned into 2-flats, all of which are accommodated by $\overline{H}$.
		
		\noindent \textit{Proof of Claim 1.} The property of whether $\overline{H}$ accommodates a given 2-flat is invariant under translations of $\overline{H}$, so we can assume that $\{\mathbf{0},\overline{\mathbf{v}}\}\subseteq \overline{H}\cap \overline{L}$ and $(\overline{\mathbf{x}}+\overline{L})\setminus \{\overline{\mathbf{x}}+\overline{\mathbf{w}}\}\subseteq \overline{H}$ with $\mathbf{v},\mathbf{w}\in L$, $\mathbf{v}\ne \mathbf{0},\mathbf{u}$. Now partition $F$ into $2^{k-1}$ pairs of sum $\mathbf{v}$, and arbitrarily pair up these pairs to form $2^{k-2}$ quadruples. Each quadruple is a $2$-flat of the form $\{\mathbf{x}+\mathbf{t}_1, \mathbf{x}+\mathbf{t}_1+\mathbf{v}, \mathbf{x}+\mathbf{t}_2, \mathbf{x}+\mathbf{t}_2+\mathbf{v}\}$ for $\mathbf{t}_1,\mathbf{t}_2\in L$ incongruent mod $\mathbf{v}$. Any such quadruple is accommodated by $\overline{H}$, as one can take the ordering $(\mathbf{x}+\mathbf{t}_1,~ \mathbf{x}+\mathbf{t}_1+\mathbf{v},~ \mathbf{x}+\mathbf{t}_2+\mathbf{v},~ \mathbf{x}+\mathbf{t}_2)$, whose corresponding partial sums in $\overline{C}$ are $(\overline{\mathbf{0}}, \overline{\mathbf{x}\vphantom{0}}+\overline{\mathbf{t}_1\vphantom{0}}, \overline{\mathbf{v}\vphantom{0}}, \overline{\mathbf{x}\vphantom{0}}+\overline{\mathbf{t}_2\vphantom{0}})$. If $\overline{\mathbf{w}}=\overline{\mathbf{t}_1}$ then swap the first two differences in the ordering, and if $\overline{\mathbf{w}}=\overline{\mathbf{t}_2}$, then swap the last two differences, hence getting a good ordering that avoids the partial sum $\overline{\mathbf{x}}+\overline{\mathbf{w}}$. $\qed$ \medskip
		
		Let us prove the following claim by induction on $k$:
        \medskip
		
		\noindent \textit{Claim 2.} Let $k\ge 2$ and $F\subseteq \ff_2^n$ be a $\mathbf{u}$-nice $k$-flat. Define $L$ and $C$ as above. Then if $\overline{H}\subseteq \overline{C}$ is a subset of size at least $2^k+2$, then $F$ can be partitioned into $2$-flats, all of which are accommodated by $\overline{H}$. Furthermore if $|\overline{H}|=2^k+1$ and there is no such partitioning, then one of the two $\overline{L}$-cosets of $\overline{C}$ is fully contained in $\overline{H}$ and the other intersects $\overline{H}$ in just one element.
		
		\noindent \textit{Proof of Claim 2.} Firstly note that in the case $k=2$, $F$ is a $2$-flat itself which is accommodated by $\overline{H}$ if and only if $C\setminus H$ is not $\mathbf{u}$-blocking for $F$. So the two parts of our statement follow from Proposition \ref{blockingratio_2flat} and Remark \ref{blockingratio_2flat_eqremark}, respectively.
		
		Now let $k\ge 3$, and suppose that $F$ has no good partitioning. For each $1\le i\le k$, define $L_i=\langle \mathbf{a}_1,..., \mathbf{a}_{i-1}, \mathbf{a}_{i+1}, ..., \mathbf{a}_k\rangle$. If there exists $i$ such that both $\mathbf{x}+L_i$ and $\mathbf{x}+\mathbf{a}_i+L_i$ can be partitioned into $2$-flats accommodated by $\overline{H}$, then $F$ can also be partitioned in this way. So suppose that for all $i$, it is impossible to partition either $\mathbf{x}+L_i$ or $\mathbf{x}+\mathbf{a}_i+L_i$ in this way. Without loss of generality, we may assume that for every $i$, $\mathbf{x}+L_i$ cannot be partitioned like this. For each $i$, clearly $\mathbf{x}+L_i$ is a $\mathbf{u}$-nice $(k-1)$-flat, for which Claim 2 can be recursively used. Let $C_i=\langle \mathbf{x},\mathbf{u}\rangle \oplus L_i$, $\overline{H_i}=\overline{H}\cap \overline{C_i}$ and $\overline{H'_i}=\overline{H}\setminus \overline{C_i}=\overline{H}\cap (\overline{\mathbf{a}_i}+\overline{C_i})$. Then by the inductive hypothesis, if $|\overline{H_i}|\ge 2^{k-1}+2$ then $\mathbf{x}+L_i$ can be partitioned into $2$-flats accommodated by $\overline{H_i}$ and hence by $\overline{H}$. Similarly, if $|\overline{H'_i}|\ge 2^{k-1}+2$ then $\mathbf{x}+L_i$ can again be partitioned into $2$-flats accommodated by $\overline{H'_i}$ (which we can work in instead of $\overline{H_i}$, as translation of a set does not affect whether it accommodates a certain $2$-flat) and hence by $\overline{H}$. So we must have $|\overline{H_i}|\le 2^{k-1}+1$ and $|\overline{H'_i}|\le 2^{k-1}+1$ for every $i$. Overall, if $F$ has no good partitioning then $\overline{H}\subseteq \overline{C}$ is of size $\le 2^k+2$.
		
		If $|\overline{H}|=2^k+2$ then for each $i$, $|\overline{H_i}|=|\overline{H'_i}|=2^{k-1}+1$, and as neither $\overline{H_i}$ nor $\overline{H'_i}$ can accommodate $\mathbf{x}+L_i$, by the induction hypothesis one of the two $\overline{L_i}$-cosets of $\overline{C_i}$ must be fully contained in $\overline{H}$ (say $\overline{\mathbf{y}_i}+\overline{L_i}$), and similarly one of the two $\overline{L_i}$-cosets of $\overline{\mathbf{a}_i}+\overline{C_i}$ must also be fully contained in $\overline{H}$ (say $\overline{\mathbf{y}'_i}+\overline{L_i}$). The other two $\overline{L_i}$-cosets of $\overline{C}$ (namely $\overline{\mathbf{y}_i}+\overline{\mathbf{a}_i}+\overline{L_i}$ and $\overline{\mathbf{y}'_i}+\overline{\mathbf{a}_i}+\overline{L_i}$) must contain only one element of $\overline{H}$ each.
		
		Now we show that $\overline{\mathbf{y}_1}$ and $\overline{\mathbf{y}'_1}$ lie in the same $\overline{L}$-coset of $\overline{C}$. Supposing otherwise, $\overline{C_2}$ intersects all four $\overline{L_1}$-cosets of $\overline{C}$ in $2^{k-2}$ elements, which means that $\overline{\mathbf{y}_1}+\overline{L_1}$ and $\overline{\mathbf{y}'_1}+\overline{L_1}$ (both contained in $H$), which together intersect all four $\overline{L_2}$-cosets of $\overline{C}$, contain at least $2^{k-2}$ elements in each $\overline{L_2}$-coset, so there cannot be any $\overline{L_2}$-cosets containing only one element of $\overline{H}$, a contradiction.
		
		Without loss of generality, $\overline{\mathbf{y}_1}, \overline{\mathbf{y}'_1}\in \overline{L}$, so $\overline{L}\subseteq \overline{H}$. Then $\overline{x}+\overline{L}$ has 2 elements of $\overline{H}$ so by Claim 1, the partitioning is possible, a contradiction.

		If $|\overline{H}|=2^k+1$ then for each $i$, one of $|\overline{H_i}|$ and $|\overline{H'_i}|$ is $2^{k-1}+1$ and the other is $2^{k-1}$. Without loss of generality, $|\overline{H_i}|=2^{k-1}+1$ for all $i$; then by the inductive hypothesis, one of the $\overline{L}$-cosets contains $2^{k-1}$ elements of $\overline{H_i}$ and the other contains just one. We can assume that $\overline{L}$ has $2^{k-1}$ elements of $\overline{H_1}$; then it automatically contains all $2^{k-2}$ elements of $\overline{H_i}\cap \overline{H_1}$ for all $2\le i\le k$, so actually $\overline{L}$ must contain $2^{k-1}$ elements of $\overline{H_i}$ for all $i$. But then $\overline{L}\cap \overline{C_i}\subseteq \overline{H}$ for each $i$, meaning that $\overline{H}$ has at least $2^k-1$ elements in $\overline{L}$. If it has only $2^k-1$ elements then the partitioning is possible by Claim 1, and if it has $2^k$ elements then we have the exceptional case stated in the Claim. $\qed$ \medskip

		Clearly Claim 2 can also be used for translates of $\overline{C}$ instead of $\overline{C}$. So let us use it for the $\overline{C}$-cosets of $\overline{V}$. If a $\overline{C}$-coset $\overline{\mathbf{y}}+\overline{C}$ satisfies $|(\overline{\mathbf{y}}+\overline{C})\cap \overline{S}|\le 2^k-2$, then by Claim 2, $F$ can be partitioned into $2$-flats all accommodated by $(\overline{\mathbf{y}}+\overline{C})\setminus \overline{S}$. Let $N_k$ be the number of ways to partition a $k$-flat into $2$-flats. As a partitioning can be encoded by giving an ordering of the vectors, $N_k\le (2^k)!\le 2^{k\cdot 2^k}$. If we have $2^{k-2}$ cosets for which $F$ is partitioned in the same way, then we can take this partitioning of $F$ and accordingly place one $2$-flat into each coset. So if the number of cosets which intersect $\overline{S}$ in at most $2^k-2$ elements is at least $N_k\cdot 2^{k-2}$, then this holds by the pigeonhole principle, and $S$ is not $\mathbf{u}$-blocking for $F$. The total number of $\overline{C}$-cosets of $\overline{V}$ is $2^{n-k-2}$, so if $S$ is $\mathbf{u}$-blocking, 
		
		$$|\overline{S}|\ge (2^{n-k-2}-N_k\cdot 2^{k-2})(2^k-1),$$
		
		and 
		
		$$|S|\ge (2^{n-k-1}-N_k\cdot 2^{k-1})(2^k-1)\ge 2^{n-1}-2^{n-k-1}-2^{k\cdot (2^k+2)},$$ 
		
		meaning that
		
		$$\beta_{\mathbf{u}}(F)\ge \frac12-\frac{1}{2^{k+1}}-\frac{1}{2^{n-k(2^k+2)}}.$$
	\end{proof}
	
	\subsection{Perfect matching in the case of many equal vectors}
	
	We are now ready to prove Theorem \ref{mainthm_halfpluseps}.
	\medskip
	
	\noindent \textbf{Theorem \ref{mainthm_halfpluseps}.} \textit{For every $\eps>0$, there exists a value $n_0$ such that for all $n\ge n_0$, the main conjecture is true in the case when at least a fraction $\frac12+\eps$ of the differences are all equal.}
	
	\begin{proof}
		Let $\mathbf{u}\in \ff_2^n$ denote the value that appears most frequently among the given differences $\mathbf{d_1}, \mathbf{d_2}, \dots, \mathbf{d_m}$. (Here the number of differences is $m=\frac12\cdot 2^n$.)
		
		Let $H$ denote the multiset of vectors $\mathbf{d_i}$ which are not equal to $\mathbf{u}$. Then $|H|<\left(\frac12-\eps\right)m$.
		
		For constructing an appropriate perfect matching of $\ff_2^n$, we will group the given differences into several classes $D_1$, $D_2$, ..., $D_{\ell}$ in such a way that Proposition \ref{construction_scheme} can be used.
		
		First we will define the grouping, then give the order of the groups, and at the end of the proof we will explain why the condition of Proposition \ref{construction_scheme} is satisfied when the classes are ordered in this way.
		
		First of all, each difference equal to $\mathbf{u}$ will form its own class, and $\beta_{\mathbf{u}}(\{\mathbf{u}\})=1$.
		
		For differences not equal to $\mathbf{u}$, first we check if there is a pair of equal differences $\mathbf{d}\ne \mathbf{u}$. As long as there is such a pair, the pair will form its own class, and $\beta_{\mathbf{u}}(\{\mathbf{d},\mathbf{d}\})=\frac12$.
		
		Let $H'$ consist of the remaining elements of $H$. Then the elements of $H'$ are all distinct and are not congruent to $\mathbf{0}$ mod $\mathbf{u}$. Furthermore, $\sum H'=|H'|\mathbf{u}$.
		
		Now divide the elements of $H'$ into two classes $H'_1$ and $H'_2$ such that none of the two classes contains a pair of elements which are congruent mod $\mathbf{u}$. This can be done in the following way, for example: consider all $\langle \mathbf{u}\rangle$-cosets which have two elements in $H'$, and for each such coset, place one element in $H'_1$ and the other in $H'_2$. For those $\langle \mathbf{u}\rangle$-cosets which have only one element in $H'$, place those elements into $H'_1$.
		
		Let $k=\left\lfloor \frac13\log_2n\right\rfloor$. In $H'_1$, repeatedly find a $k$-flat and put its elements into a new class, removing them from $H'_1$. (Each $k$-flat forms its own distinct class of $2^k$ differences.)  Do this until no more $k$-flats can be found among the remaining elements of $H'_1$. Then do the same for $H'_2$. By Proposition \ref{blockingratio_kflat}, each class $D_i$ so obtained satisfies $\beta_{\mathbf{u}}(D_i)\ge \frac12-\frac{1}{2^{k+1}}-\frac{1}{2^{n-k(2^k+2)}}$.
		
		Now by Proposition \ref{kflat_findable}, the number of remaining elements in $H'_1$ will be less than $2\cdot~2^{n\left(1-\frac{1}{2^{k-1}}\right)}$, and the same is true for $H'_2$. Let $H''$ consist of all remaining elements from $H'_1$ and $H'_2$; then $|H''|< 4\cdot 2^{n\left(1-\frac{1}{2^{k-1}}\right)}$.
		
		Note that every $k$-flat has an even number of elements and $\mathbf{0}$ sum, so $|H''|\equiv |H'|\pmod{2}$ and $\sum H''=\sum H'$, meaning that $\sum H''=|H''|\mathbf{u}$ holds. $H''$ does not contain any element congruent to $\mathbf{0}$ mod $\mathbf{u}$.
		
		By Proposition \ref{ugood_partition}, $H''$ can be partitioned into $\mathbf{u}$-good multisets. (Here they are actually sets, since the elements of $H''$ are all distinct.) By Remark \ref{ugood_partition_sizeremark}, every such multiset has size $\le 2n$. Each $\mathbf{u}$-good multiset will form its own class. By Proposition \ref{beta_greedybound}, every such class $D$ will have $\beta_{\mathbf{u}}(D)\ge \frac{1}{\binom{2n}{2}+1}\ge \frac{1}{2n^2}$.
		
		Now order the classes in the following way: first put the $\mathbf{u}$-good multisets from $H''$, then the $k$-flats from $H'\setminus H''$, then the pairs $\{\mathbf{d},\mathbf{d}\}$ and finally the singleton classes $\{\mathbf{u}\}$. Then the condition of Proposition \ref{construction_scheme} will be satisfied:
		
		\begin{itemize}
			\item For a $\mathbf{u}$-good class $D_i$, we have $\sum_{j=1}^{i-1} |D_j|<|H''|<4\cdot 2^{n\left(1-\frac{1}{2^{k-1}}\right)}$. So
			
			$$\frac{\sum_{j=1}^{i-1} |D_j|}{2^{n-1}}<8\cdot 2^{-\frac{1}{2^{k-1}}n}<\frac{1}{2n^2}\le \beta_{\mathbf{u}}(D_i),$$
			
			since $$16n^2<2^{\frac{1}{2^{k-1}}n} ~\Leftrightarrow~ 2\log_2 n+4<\frac{1}{2^{k-1}}n~\Leftrightarrow~ 2^{k-1}<\frac{n}{2\log_2 n+4},$$
			
			which is true because $2^{k-1}\le 2^{\frac13\log_2 n-1}=\frac12n^{1/3}\le \frac{n}{2\log_2 n+4}$ for $n$ large enough.
			
			\item For a $k$-flat of $H'\setminus H''$, we have $\sum_{j=1}^{i-1} |D_j|<|H'|<(\frac12-\eps)\cdot 2^{n-1}$, so $\frac{\sum_{j=1}^{i-1} |D_j|}{2^{n-1}}<\frac12-\eps$, and using the fact that $k+1\ge \frac13\log_2 n$,
			
			$$\beta_{\mathbf{u}}(D_i)\ge \frac12-\frac{1}{2^{k+1}}-\frac{1}{2^{n-k(2^k+2)}}\ge \frac12-\frac{1}{n^{1/3}}-\frac{1}{2^{n-\frac13\log_2 n(n^{1/3}+2)}}>\frac12-\eps$$ for $n$ large enough.
			
			\item For a pair $\{\mathbf{d}, \mathbf{d}\}$, we have $\sum_{j=1}^{i-1} |D_j|<|H|<(\frac12-\eps)\cdot 2^{n-1}$, so $$\beta_{\mathbf{u}}(D_i)=\frac12>\frac12-\eps>\frac{\sum_{j=1}^{i-1} |D_j|}{2^{n-1}}.$$
			
			\item For singletons $\{\mathbf{u}\}$, we have $\sum_{j=1}^{i-1} |D_j|<2^{n-1}$, so $\beta_{\mathbf{u}}(D_i)=1>\frac{\sum_{j=1}^{i-1} |D_j|}{2^{n-1}}$.
		\end{itemize}
		
		So an appropriate perfect matching of $\ff_2^n$ exists according to the given differences.
	\end{proof}
	
	\section*{Acknowledgements}
        I am very grateful to Zoltán Lóránt Nagy for his helpful comments and advice regarding my research and this paper itself.

        I would like to thank Péter Csikvári for bringing the conjecture of Balister, Győri and Schelp to my attention, and for his helpful comments and advice regarding my research.

	    I would also like to thank Ago-Erik Riet for calling my attention to the coding-theoretical literature on batch codes related to the main conjecture, and the anonymous referee for their valuable remarks on the paper.

	\section*{Appendix}
	
	Here we finish the remaining cases of the following Proposition.
	\medskip
	
	\noindent \textbf{Proposition \ref{merging_circuits}.} \textit{Let $G$ be an abelian group, and let $A=(a_1,a_2,...,a_k)$ and $B=(b_1,b_2,...,b_{\ell})$ be zero-sum sequences of elements in $G$, where $2\le k\le \ell$ are integers. Suppose that $A.B$ is free of signed $<k$-sums and $B$ is free of signed $<\ell$-sums. Then $A.B$ is simply orderable, except in the following two special cases:}
	
	\begin{itemize}
		\item \textit{$k=\ell=2$, $A\sim B\sim (a,-a)$ for some $0\ne a\in G$,}
		
		\item \textit{$k=2$, $\ell=3$, $A\sim (a,a)$ and $B\sim (a,d,a-d)$ for some $a,d\in G$ with $a,d\ne 0$, $2a=0$, $d\ne a$ and $a+2d\ne 0$.}
	\end{itemize}
	
	\begin{proof}[Proof of Proposition \ref{merging_circuits}: cases of small $k$ and $\ell$]
			
		In the remaining cases we have $2\le k\le \ell\le 5$. For these cases, the following observation will help: for any distinct $U,U'\subseteq [k+\ell]$, if $\sum (A.B)_U=\sum (A.B)_{U'}=0$ then $|U\triangle U'|\ge k$. (Otherwise, the indices corresponding to $U\triangle U'$ would give a $<k$-element signed sum of zero.)
		
		Therefore the characteristic vectors of subsets $U\subseteq [k+\ell]$ corresponding to zero sums form a binary code of length $k+\ell$ and minimum distance at least $k$. By the Hamming bound, the size $N$ of such a code satisfies 
		
		$$N\le \frac{2^{k+\ell}}{\sum\limits_{i=0}^t \binom{k+\ell}{i}} \eqno (*)$$ 
		
		where $t=\left\lfloor \frac{k-1}{2}\right \rfloor$. Here $N$ is the number of subsets $U$ of $[k+\ell]$ with $\sum U=0$, so $N=N_0+N_1+...+N_{k+\ell}=2+N_k+N_{k+1}+...+N_{\ell}$.
		
		\textbf{Case $k=\ell=5$:} In this case, by $(*)$, $N=2+N_5\le \frac{2^{10}}{\binom{10}{0}+\binom{10}{1}+\binom{10}{2}}<19$. So $N_5\le 16$, and $\alpha=5\cdot \frac{N_5}{\binom{10}{5}}\le \frac{20}{63}<1$.
		
		\textbf{Case $k=4$, $\ell=5$:} Let us consider $N_4$. For every $S\subseteq [4]$, there are at most two subsets $T\subseteq [5]$ with $A_S.B_T=0$, and if there are two then they are complements of each other, and hence they have different sizes. So there is at most one $T$ such that $\sum A_S.B_T$ is zero-sum and $|S|+|T|=4$. For every $U,U'\subseteq [9]$ with $\sum (A.B)_U=\sum (A.B)_{U'}=0$, either $U=U'$ or $|U\triangle U'|\ge 4$. So supposing that $|U|=|U'|=4$ with $U\ne U'$, this means that $|U\cap U'|\ne 3$, and as complements of zero-sum sequences are also zero-sum, $|U\triangle \overline{U'}|\ge 4$ holds too, meaning that $|U \cap U'|\not\in \{0,1\}$ either, so $|U\cap U'|=2$. Take $U'=\{1,2,3,4\}$. Now either $U=\{1,2,3,4\}$ or $U$ has an intersection with $\{1,2,3,4\}$ of size $2$, which can be extended to a zero-sum set of size $4$ in at most one way, meaning that $N_4\le 7$ altogether. This means $\alpha=\frac{9}{2}\cdot \frac{2N_4}{\binom{9}{4}}\le \frac12<1$.
		
		\textbf{Case $k=\ell=4$:} For each $S\subseteq [4]$ there are at most two $T\subseteq [4]$ with $\sum A_S.B_T=0$, and similarly to the previous case, one can obtain that any $U,U'\subseteq [8]$ of size $4$ and zero sum cannot have an intersection of size 3. $\{1,2,3,4\}$ and $\{5,6,7,8\}$ are zero-sum, so by investigating the intersection sizes of $U$ with these two sets, this means that any other zero-sum $U$ has exactly two elements from $\{1,2,3,4\}$, and any two such elements can be extended to a zero-sum set of size $4$ in at most two ways, giving $N_4\le 2+\binom{4}{2}\cdot 2=14$. So $\alpha=4\cdot \frac{N_4}{\binom{8}{4}}\le \frac45<1$.
		
		\textbf{Case $k=3$, $\ell=5$:} For each $S\subseteq [3]$ there are at most two $T\subseteq [5]$ with $\sum A_S.B_T=0$ (and if there are two then they are complements of each other), and any zero-sum $U,U'\subseteq [8]$ are equal or have Hamming distance at least 3. So if $|U|=|U'|=3$ then $|U\cap U'|=2$ is impossible, and by taking the complement of $U'$, $|U\cap U'|=0$ is also impossible. If $|U|=3$ and $|U'|=4$ then $|U\cap U'|=3$ is impossible, and by taking the complement of $U'$, $|U\cap U'|=0$ is also impossible. Fixing $U=\{1,2,3\}$, this means that $3$-element zero-sum index sets have intersection $1$ or $3$ with $U$, and every such intersection can only occur once (meaning that $N_3\le 4$), and $4$-element zero-sum index sets have intersection $1$ or $2$ with $U$, and every such intersection can only occur once (meaning that $N_4\le \binom{3}{1}+\binom{3}{2}=6$). Therefore $\alpha=4\left(\frac{2N_3}{\binom{8}{3}}+\frac{N_4}{\binom{8}{4}}\right)\le 4\left(\frac17+\frac{3}{35}\right)=\frac{32}{35}<1$.
        
		\textbf{Case $k=3$, $\ell=4$:} Using the techniques seen above, one obtains that for any zero-sum $U\subseteq [7]$ of size $3$, either $U=\{1,2,3\}$ or $|U\cap \{1,2,3\}|=1$, and one-element subsets of $\{1,2,3\}$ can be extended to zero-sum sets in $[7]$ of size $3$ in at most two ways. This gives altogether $N_3\le 7$, however we need $N_3\le 4$ to show that $\alpha=\frac72\cdot \frac{2N_3}{\binom{7}{3}}=\frac{N_3}{5}<1$. So supposing $N_3\ge 5$, there is a 1-element subset of $\{1,2,3\}$ that can be extended in two ways. We can assume that this is $\{1\}$. As three-element zero-sum subsets cannot intersect in exactly 2 elements (or else they would violate the minimum distance of $\ge 3$), one can assume that $\{1,4,5\}$ and $\{1,6,7\}$ correspond to zero-sum subsets, meaning that $a_1+b_1+b_2=a_1+b_3+b_4=0$. Still avoiding intersections of size 2, without loss of generality, we can also assume that $\{2,4,6\}$ is a further zero-sum subset, so $a_2+b_1+b_3=0$ too. But then $(a_1, a_2, b_3, a_3, b_1, b_2, b_4)$ can be seen to be a good cyclic ordering, as each subinterval of length 3 has an intersection of size $2$ with a known zero-sum subset of size $3$ ($\{a_1,a_2,a_3\}$, $\{a_1,b_1,b_2\}$, $\{a_1,b_3,b_4\}$ or $\{a_2,b_1,b_3\}$), and hence does not itself have zero sum.
		
		\textbf{Case $k=\ell=3$:} Here $\{1,2,3\}$ and $\{4,5,6\}$ are zero-sum, and two zero-sum triples cannot have an intersection of size 2, meaning that there cannot be any further zero-sum triples in $[6]$, and $N_3\le 2$, meaning that $\alpha=3\cdot \frac{N_3}{\binom{6}{3}}\le \frac{3}{10}<1$.
		
		\textbf{Case $k=2$, $\ell=5$:} Let $A=(a,-a)$. We need $\alpha=\frac72\left(\frac{2N_2}{\binom{7}{2}}+\frac{2N_3}{\binom{7}{3}}\right)<1 \Leftrightarrow 5N_2+3N_3<15$. By Observation 1, $N_2+N_3+N_4+N_5=2(N_2+N_3)\le 6$, so $N_2+N_3\le 3$. Hence $5N_2+3N_3\le 5(N_2+N_3)\le 15$ with equality if and only if $N_2=3$ and $N_3=0$. But $B$ is free of signed $<5$-sums, so at most one of its elements can belong to the set $\{a, -a\}$, and there is no zero-sum pair within $B$. This implies that the number of zero-sum pairs is at most $2$, so $N_2=3$ cannot happen, unless $2a=0$. However in this case $A=(a,a)$ and $B=(a,d,e,f,g)$, and $(a,d,a,e,a,f,g)$ is a good cyclic ordering (as using $2a=0$, every two- or three-element subinterval along this cycle sums to a sum of $1$, $2$ or $3$ distinct elements of $B$ which must be nonzero).
		
		\textbf{Case $k=2$, $\ell=4$:} Let $A=(a,-a)$ and $B=(c,d,e,f)$. We need $\alpha=3\left(\frac{2N_2}{\binom{6}{2}}+\frac{N_3}{\binom{6}{3}}\right)<1 \Leftrightarrow 8N_2+3N_3<20$. By Observation 1, $N_2+N_3+N_4=2N_2+N_3\le 6$. If $N_2=0$ then $N_3\le 6$ so $3N_3<20$. Again, as $B$ has no signed $<4$-sums, at most one of its elements can belong to $\{a, -a\}$, and there is no zero-sum pair within $B$. Distinguish three cases depending on the value of $N_2$.
		
		If $N_2=1$ then $\{a,-a\}$ is the only zero sum of size 2. Both $\{a\}$ and $\{-a\}$ might be extended to a 3-element zero-sum multiset in at most two ways, giving $N_3\le 4$. If $8N_2+3N_3<20$ we are done, so we can assume $N_3=4$. Then $\{a\}$ and $\{-a\}$ extend to 3-element zero-sum multisets in two ways each. If $\{a,c,d\}$ has zero sum then for example $\{a,c,e\}$ cannot have zero sum, as then we would have $d=e$. A complement of a zero-sum set is also zero-sum, so the only case is that the 3-element zero-sums are $a+c+d=a+e+f=(-a)+c+d=(-a)+e+f=0$, meaning that $a=-a$ so $2a=0$. In this case, $(a,e,d,a,f,c)$ is a good cyclic ordering of the six elements: it doesn't have 2- or 3-element zero-sum subintervals, as we have enumerated all of those and they do not appear.
		
		If $N_2=2$ then $A=(a,-a)$, $B=(a,d,e,f)$ where $a\ne -a$ (so $2a\ne 0$). For $N_3\le 1$ we are then done, and for $N_3\ge 2$, $\{a\}$ in $A$ must extend to three-element zero-sum multisets in two ways, and $N_3=2$. So we can assume $a+a+f=a+d+e=0$, meaning that $B=(a, d, a-d, -2a)$. Since $B$ is free of signed $<4$-sums, we have $a,d,2a,3a\ne 0$, $a\ne 2d$ and $d\ne \pm a, \pm 2a, \pm 3a$. Then $(-a, a-d, a, d, a, -2a)$ is a good cyclic ordering, which is easy to check.
		
		If $N_2=3$ then we must have $a=-a$ (i.e., $2a=0$), and $A=(a,a)$ and $B=(a,d,e,f)$, and then $(a,d,a,e,a,f)$ is a good cyclic ordering, similarly to the case seen for $(k,\ell)=(2,5)$.
		
		\textbf{Case $k=2$, $\ell=3$:} Let $A=(a,-a)$ and $B=(c,d,e)$. If $N_2\le 1$ then $\alpha=\frac52\cdot \frac{2N_2}{\binom{5}{2}}\le \frac12<1$, so $N_2\ge 2$. Similarly to the previous cases, at most one element in $B$ can be $\pm a$, giving $N_2=2$ if $a\ne -a$ and $N_2=3$ if $a=-a$. In the former case, $A=(a,-a)$ and $B=(a, d, -a-d)$ where the ordering $(-a, -a-d, a, a, d)$ works. In the latter case, we have $A=(a,a)$ and $B=(a, d, a-d)$ where $a,d\ne 0$, $d\ne a$ and $a+2d\ne 0$, since $B$ is free of signed $<3$-sums. In this case there is no good cyclic ordering, as 3 out of the 5 elements are equal to $a$, and two of them have to appear consecutively.
		
		\textbf{Case $k=\ell=2$:} Let $A=(a,-a)$ and $B=(b,-b)$. If $b\ne \pm a$ then $(a,b,-a,-b)$ is a good cyclic ordering, otherwise it is impossible to order $A.B=(a,-a,a,-a)$ simply.
	\end{proof}
\end{document}